\documentclass[12pt,fullpage]{article}
\usepackage{amsmath}
\usepackage{amsfonts,amssymb}
\usepackage{amsthm}
\usepackage{graphics,graphicx,pspicture,graphpap}
\usepackage{hyperref}
\usepackage[right = 2.5cm, left=2.5cm, top = 2.5cm, bottom =2.5cm]{geometry}
\usepackage{color}
\usepackage{pstricks,pst-node,pst-text,pst-3d}

\numberwithin{equation}{section}

\newcommand{\Real}{\mathbb R}

\newcommand{\Torus}{\mathbb T} 
\newcommand{\norm}[1]{\|#1\|}
\newcommand{\abs}[1]{\left\vert#1\right\vert}
\newcommand{\set}[1]{\left\{#1\right\}}

\newcommand{\grad}{\nabla}

\newtheorem{theorem}{Theorem}
\theoremstyle{lemma}

\theoremstyle{definition}

\newtheorem{remark}{Remark}

\theoremstyle{lemma}
\newtheorem{lemma}{Lemma}

\begin{document}
\title{Blister patterns and energy minimization in compressed thin films on compliant substrates} 

\author{Jacob Bedrossian\footnote{Partially supported by NSF Postdoctoral Fellowship in Mathematical Sciences, DMS-1103765}, Robert V. Kohn\footnote{Partially supported by NSF grants DMS-0807347 and OISE-0967140} \\
Courant Institute of Mathematical Sciences\\
New York University\\
{\tt jacob@cims.nyu.edu} and {\tt kohn@cims.nyu.edu}}

\date{\today}

\maketitle
\begin{abstract}
This paper is motivated by the complex blister patterns sometimes seen in thin
elastic films on thick, compliant substrates. 
These patterns are often induced by an elastic misfit which compresses the film. 
Blistering permits the film to expand locally, reducing the
elastic energy of the system. 
It is therefore natural to ask: what is the minimum elastic energy achievable by blistering on a fixed area fraction of the substrate? This is a variational problem involving both the {\it elastic deformation} of the film and substrate and
the {\it geometry} of the blistered region. It involves three small parameters: the
{\it nondimensionalized thickness} of the film, the {\it compliance ratio} of the film/substrate
pair and the {\it mismatch strain}. In formulating the problem, we use a small-slope (F\"oppl-von K\'arm\'an) approximation
for the elastic energy of the film, and a local approximation for the elastic energy of the substrate.

For a 1D version of the problem, we obtain ``matching'' upper and lower
bounds on the minimum energy, in the sense that both bounds have the same scaling behavior with respect to the small parameters. 
The upper bound is straightforward and familiar: it is achieved by periodic blistering on a specific length scale.
The lower bound is more subtle, since it must be proved without any assumption on the
geometry of the blistered region. 

For a 2D version of the problem, our results are less complete.
Our upper and lower bounds only ``match'' in their scaling with respect to the nondimensionalized
thickness, not in the dependence on the compliance ratio and the mismatch strain. 
The lower bound is an easy consequence of our one-dimensional analysis. 
The upper bound considers a 2D lattice of blisters, and uses ideas from the literature on the folding or ``crumpling'' of a
confined elastic sheet.
Our main 2D result is that in a certain parameter regime, the elastic energy
of this lattice is significantly lower than that of a few large blisters. 

\end{abstract}

\section{Introduction}
There are many technologically important systems involving thin films bonded to thick
substrates. Two specific applications are flexible electronics (see e.g. \cite{VellaEtAl09}
and the references there) and microfluidic devices (see e.g. \cite{VolinskyWatersWright05}). 

Films often experience compression, due e.g. to elastic mismatch or substrate curvature.
The consequences of compression vary widely. In some systems the film buckles but stays
bonded to the substrate (see e.g. \cite{AudolyBoudaoudIII08,CaiEtAl11,HuangHongSuo04,HuangHongSuo05,SongJiangEtAl08,KohnNguyen12}).
In other systems the film delaminates, i.e. forms blisters. While
the morphology of blistering varies considerably, many systems fall into one of the following categories:
\begin{enumerate}
\item[(i)] those that form a few large blisters, comparable in scale to the size of the substrate; and
\item[(ii)] those that form web-like or lattice-like two-dimensional blister patterns, with a length scale much smaller than
that of the substrate. 
\end{enumerate}
See e.g. \cite{GiolaOrtiz97} and the references there, as well as the more recent
articles \cite{MoonEtAl02,MoonJensenEtAl02,YuZhangChen09}. Specifically, see figures 3,4,7,10 in \cite{GiolaOrtiz97} for examples of web-like blister patterns and figures 1 and 14 in \cite{GiolaOrtiz97} for large isolated blisters. 

The driving force behind both buckling and blistering is a tendency to reduce the overall elastic
energy. Briefly: the film is compressed by the substrate; buckling reduces the compression by permitting the
film to expand; blistering reduces the compression still further, by eliminating the source
of compression in the interior of the blister.

Blisters grow incrementally; the process is similar to fracture in which dynamics plays a role \cite{SuoHutchinson90,HutchinsonThoulessLiniger92,VellingaEtAl08}.  
Moreover, the evolution of blistering depends on history \cite{HuangHongSuo04,HuangHongSuo05} as well as on imperfections in the system \cite{MoonEtAl04}. 
However, energy minimization is the driving force, so it is natural to ask:
\begin{enumerate}
\item[(a)] What buckling or blistering patterns are most effective in decreasing the
energy of the system?
\item[(b)] How much can the energy be decreased by this mechanism?
\end{enumerate}
Versions of these questions have been addressed for buckling without blistering
\cite{AudolyBoudaoudIII08,KohnNguyen12}, and for a blister of fixed geometry on a rigid substrate 
\cite{BelgacemContiSimoneMuller00,BelgacemContiSimoneMuller02,JinSternberg01}.

The present article addresses the same questions (a)-(b) in a different setting. Specifically: we address
the case where blistering is permitted and the geometry of the blistered region is part of the optimization.
We treat the area fraction of blistering as a parameter, fixed in advance. Our formulation of the
elastic energy is presented in Section \ref{sec:FvK} -- the final form is (\ref{def:VarP}) -- but briefly:
\begin{itemize}
\item In modeling the membrane and bending energy of the film we use a
small-slope (F\"oppl-von K\'arm\'an) approximation. Moreover, motivated by the
numerical work of Jagla \cite{Jagla07}, we assume the out-of-plane deformation
vanishes in the bonded region. 
\item In modeling the elastic energy of the substrate we use a local approximation, chosen so that for
periodic blister patterns it has the same scaling as the elastic energy of a semi-infinite substrate.
\end{itemize}
We mention in passing that blistering sometimes leads to material failure, e.g. via cracking along
ridges \cite{FaulhaberEtAl06,MoonJensenEtAl02,VellingaEtAl08}. In this paper, however, we
ignore the possibility of material failure: our films are always elastic.

Our problem has three small parameters: the {\it nondimensionalized thickness}, 
the {\it compliance ratio} and the {\it mismatch strain}. Our main focus is on how the energy scales with respect to these
parameters when they are sufficiently small. (Our focus is thus not on the behavior near a bifurcation;
rather, to use the terminology of \cite{DavidovitchEtAl11,DavidovitchEtAl12}, our interest lies in the
``far from threshold'' parameter regime.)
\bigskip 

Our results are stated with precision in Section \ref{sec:Statement-of-Results}. Summarizing
them briefly:
\begin{itemize}
\item {\it For a 1D version of the problem} quite similar to that considered
experimentally and theoretically in \cite{VellaEtAl09}, we obtain ``matching'' upper and lower
bounds on the minimum energy. They match in the sense that both bounds
have the same scaling behavior with respect to the small parameters. The upper bound is
familiar and easy: it is associated with periodic blistering on a well-chosen length scale
(consistent with the experiments reported in \cite{VellaEtAl09}). The lower bound is more
subtle, since it must be proved without any assumption on the geometry of the blistered set.

\item {\it For a 2D version of the problem} capable of representing lattice-like blister patterns,
our results are less complete. Our upper and lower bounds ``match'' in their scaling with respect
to the nondimensionalized thickness, but not in their dependence on the compliance ratio and mismatch strain. The
lower bound is an easy consequence of our one-dimensional analysis. The upper bound considers three alternatives:
a flat film; a 2D lattice of blisters with a well-chosen length scale; and a single large blister.
The most subtle part of our 2D discussion is our estimation of the energy of a 2D lattice of blisters.
It uses ideas from the literature on the folding or ``crumpling'' of a confined elastic sheet, specifically
a construction similar to that of the ``minimal ridge'' considered in
\cite{ContiMaggi08,Lobkovsky96,LobkovskyWitten97,Venkataramani03}.
\end{itemize}

One of the major questions in this area is: why do some systems form a lattice of blisters,
while others form a few large blisters? Our 2D results provide some insight, by identifying a parameter regime
in which a 2D lattice of blisters (with a well-chosen length scale) has smaller energy than a single large
blister (occupying the same area fraction of the substrate).

Another key question is: when a blister lattice forms, why do the individual blisters have a
``telephone-cord'' morphology? We have nothing new to contribute here. Perhaps the answer lies in
the mechanism by which blisters grow: a recent numerical study \cite{FaouEtAl12} shows that even
an isolated blister tends to grow with a telephone-cord morphology (see also \cite{Jagla07}).
Thinking variationally: our test function representing the 2D blister lattice uses blisters with
straight sides, but it is surely not even a local minimum. The results in \cite{Audoly99,ParryEtAl09}
suggest that straight-sided blisters can lower their energy further by developing a
less regular, telephone-cord-like morphology. We doubt, however, that this would change the
energy scaling law. We are not the first to take the view that blisters with straight sides provide
informative test functions; see e.g. the studies \cite{YuHutchinson02,CotterellChen00} concerning
substrate/film systems with different compliance ratios. The paper \cite{GiolaOrtiz97} includes an
energy-based explanation of the telephone-cord morphology, but the model used there is quite different
from ours (in particular it ignores the in-plane deformation of the film, and it takes the substrate to
be rigid).

As already noted above, our focus is on the scaling of the minimum energy with respect to the
nondimensionalized thickness, the compliance ratio and the mismatch strain -- the small parameters of this problem.
Our method is to prove upper bounds using well-chosen test functions, and lower bounds using
geometry-independent arguments. Other papers taking a similar viewpoint on problems involving
thin films include
\cite{BelgacemContiSimoneMuller00,BelgacemContiSimoneMuller02,BellaKohnCpam,BellaKohnMetric,BrandmanEtAl,ContiMaggi08,JinSternberg01,KohnNguyen12,Venkataramani03}.
For work on thin films with a similar viewpoint but providing mainly upper bounds, see
\cite{AudolyBoudaoudIII08,DavidovitchEtAl11,DavidovitchEtAl12,Lobkovsky96,LobkovskyWitten97}. A similar
viewpoint has also been applied to many other problems; for examples and some pointers to the
literature, see \cite{KohnICM}.

We are interested in the energetics of blister patterns. To avoid artifacts associated with boundaries, we use periodic boundary conditions with period $L$ (an arbitrary fixed size). Since $L$ is arbitrary, it is important that our estimates depend on it only through the non-dimensionalized film thickness $t/L$, where $t$ is the thickness of the film.  

\subsubsection*{Notations and Conventions} 
Studying periodic patterns is mathematically equivalent to working on the torus.  
Hence in what follows we denote the N-dimensional torus of length $L$ by $\Torus_L^N$ ($N = 1$ or $2$); we also use the convention $\Torus^N := \Torus^N_1$. 
For a (measurable) set $E \subset \Torus_L^N$ we write $\abs{E}$ for its ($N$-dimensional) measure. 
To avoid unnecessary clutter in formulas, we use the notation $f \lesssim g$ if there exists a constant $C > 0$
which is independent of the parameters of primary interest such that $f \leq Cg$; similarly, $f \sim g$ means there exists a constant $C > 0$ (also independent of the parameters of primary interest) such that $\frac{1}{C}f \leq g \leq Cf$. 

\subsection{F\"oppl-von K\'arm\'an Energy} \label{sec:FvK}
In this section we justify the energy we will be studying by examining several related problems, starting with the F\"oppl-von K\'arm\'an approximation for a totally bonded film. 
From there we discuss the F\"oppl-von K\'arm\'an approximation of a film that is permitted to debond, then we introduce several reductions to make the problem more tractable for rigorous analysis. 

Consider an $L\times L$ film with thickness $t$ and periodic boundary conditions, totally bonded to a substrate.
Let $u:\Torus_L^2 \rightarrow [0,\infty)$ be the out-of-plane displacement and $w:\Torus_L^2 \rightarrow \Real^2$ the in-plane displacement. 
The F\"oppl-von K\'arm\'an thin plate approximation (see e.g. \cite{AudolyPomeau}) suggests, after normalizing the bending modulus of the film to be one, an energy per unit area of the following form 
\begin{align} 
E_1[u,w] & = \frac{\alpha_m t}{L^2} \int_{[0,L]^2} \abs{e(w) + \frac{1}{2}\grad u \otimes \grad u - \eta I}^2 dx + \frac{t^3}{L^2}\int_{[0,L]^2} \abs{D^2 u}^2 dx + \frac{1}{L^2}\mathcal{E}_s[u,w], \label{def:EnonlocBonded}  
\end{align}  
where the small parameter $\eta > 0$ is the compressive {\it mismatch strain} between the substrate and film. 
We have set the Poisson's ratio to zero to simplify the arithmetic at no loss of generality: 
since we are only interested in scaling laws, all Hooke's laws are equivalent up to a computable constant. 
We will refer to the first term as the \emph{membrane energy}; it accounts for the resistance to stretching and compression in the film. 
The dimensionless constant $\alpha_m$ warrants some discussion. 
In classical thin plate theory, since we normalized the bending modulus to one and set the Poisson's ratio to zero, in reality $\alpha_m = 12$. 
However, we have chosen to denote $\alpha_m$ as a parameter, as it will indicate how large a role the membrane energy plays in the constructions and estimates, though it is important to keep in mind that it is a fixed $O(1)$ constant.    
The second term is the \emph{bending energy}, with the bending modulus normalized to one. 
The final term is the \emph{substrate energy}, which is the elastic energy in the substrate. As we discuss shortly, due to the normalization, it contains the third small parameter which is the ratio of the Young's moduli of the substrate and the film. 

The substrate is approximated as infinitely deep and treated as linearly elastic in the region $(x,y,z) \in \Torus_L^2 \times (-\infty,0]$ (also with Poisson's ratio zero without loss of generality). 
In \cite{AudolyBoudaoudI08} it is justified in the case of a relatively compliant substrate 
that the proper boundary condition to enforce at $z = 0$ is that of continuous displacement rather than traction.
As this is the case we will largely be interested in, we will take $\mathcal{E}_s$ to be the total elastic energy in $\Torus_L^2 \times (-\infty,0]$ with imposed displacement at $z=0$ given by the thin film displacement. 
It is classical that this energy is comparable to the $\dot{H}^{1/2}$ norm, defined on mean-zero functions via
\begin{align*} 
\norm{f}_{\dot{H}^{1/2}}^2 := \sum_{k\neq 0} \abs{k}\abs{\hat{f}(k)}^2, 
\end{align*}
where $\hat{f}(k) := \int_{\Torus^N_L} e^{2\pi i \frac{k}{L}\cdot x} f(x) dx$. 
By comparable we mean, 
\begin{align}
\mathcal{E}_s[w,u] \sim \alpha_s \norm{u}^2_{\dot{H}^{1/2}} + \alpha_s \norm{w}^2_{\dot{H}^{1/2}},  \label{def:BondedSubst}
\end{align}
where, due to the normalization we are taking, the dimensionless parameter $\alpha_s$ is comparable to the Young's modulus of the substrate divided by the Young's modulus of the film. We call it the {\it compliance ratio}; indeed, it reflects the relative stiffness of the substrate versus the film: $\alpha_s \gg 1$ corresponds to a substrate which is much stiffer than the film and $\alpha_s \ll 1$ corresponds to a substrate which is relatively compliant. 
In the case of partial delamination, we show that blister lattice patterns can achieve lower energies than large, isolated blisters in the compliant regime, and that this is optimal in 1D.  

The energy \eqref{def:EnonlocBonded} with \eqref{def:BondedSubst} (or something similar) has been studied in many works on pattern formation in films without blistering, for instance 
\cite{HuangHongSuo04,HuangHongSuo05,SongJiangEtAl08,CaiEtAl11,AudolyBoudaoudI08,AudolyBoudaoudII08,AudolyBoudaoudIII08,VellaEtAl09,KohnNguyen12}.
For bonded films, the optimal energy scaling law was identified in \cite{KohnNguyen12}.
Deriving an upper bound for the 1D version is a well-known computation, which we repeat here, taking $L = 1$ for simplicity. 
One can eliminate the membrane energy completely with the choice
\begin{align*} 
w(x) = \frac{\eta l}{4}\sin\left(\frac{4x}{l}\right), \;\;\; u(x) = \sqrt{\eta} l\cos\left(\frac{2x}{l}\right), 
\end{align*} 
for any $l > 0$. For such $w$ and $u$, the first term in \eqref{def:BondedSubst} scales like $O(\alpha_s \eta l)$ whereas the latter scales as $O(\alpha_s \eta^2 l)$, which is smaller by a factor of $\eta$. The bending gives a contribution of $O(t^3 \eta l^{-2})$ and optimizing in $l$ gives gives an energy which scales like $O(\alpha_s^{2/3}\eta t)$ and a well-defined optimal length-scale.

Having discussed the bonded case, we now turn to the focus of this work: the case in which the film is permitted to partially blister from the substrate.
We denote by $\Omega \subset [0,L]^2$ the region where the film remains bounded, and $\theta = \abs{\Omega}/L^2$ the area fraction. 
In our work we will treat $\theta$ as a parameter.
The energy per unit area in the partially blistered case differs from \eqref{def:EnonlocBonded} only in that the substrate energy now depends on $\Omega$, 
\begin{align} 
E_2[u,w,\Omega] & = \frac{\alpha_mt}{L^2}  \int_{[0,L]^2} \abs{e(w) + \frac{1}{2}\grad u \otimes \grad u - \eta I}^2 dx + \frac{t^3}{L^2}\int_{[0,L]^2} \abs{D^2 u}^2 dx + \frac{1}{L^2}\mathcal{E}_s[w,u,\Omega]. \label{def:EnonlocDimensional}  
\end{align}  
As before, the substrate energy $\mathcal{E}_s$ arises from a linear elasticity problem in the half-space $\Torus^2 \times (-\infty,0]$, however, now continuity of displacements is only enforced on the bonded region $\Omega$. 
In the debonded region the substrate instead satisfies traction-free boundary conditions.

Since blistering gives the system additional freedom, we expect that the energy \eqref{def:EnonlocDimensional} should be able to reach lower energies than \eqref{def:EnonlocBonded}. 
Consider again the one dimensional case with $L = 1$. 
Imagine a periodic profile with alternating blisters and bonded regions with characteristic length-scale $l$ which additionally satisfies: 
\begin{itemize} 
\item[(a)] $w_x = \eta$ and $u = 0$ in $\Omega$ and, 
\item[(b)] $\int_0^1 \abs{w_x + \frac{1}{2}u_x^2 - \eta}^2 dx = 0$. 
\end{itemize} 
Below in Section \ref{sec:UppBd1D} we make this construction explicit; see also Figure \ref{fig:Periodic1D}. Similar profiles have also been studied previously and observed in experiments \cite{VellaEtAl09}. 
Bending only occurs in the blisters and we will show that it scales as $O(t^3 \eta l^{-2})$. The substrate energy arises only from the shear imparted by $w$ which we will show in Section \ref{sec:UppBd1D} produces an energy which scales as $O(\alpha_s \eta^2 l)$. 
Optimizing in $l$ gives an energy which scales like $O(\alpha_s^{2/3}\eta^{5/3}t)$. Since $\eta$ is a small parameter, the energy is indeed smaller than the totally bonded case. 
The gain in $\eta$ came from the fact that the film no longer deflects out-of-plane in the bonded regions, which would typically cost more in the substrate energy than in-plane shear. See Remark \ref{rmk:CompBond} in Section \ref{sec:Statement-of-Results} for the analogous discussion in 2D. 

The above discussion motivates us to propose a simpler model in which $u = 0$ in $\Omega$ is imposed as a constraint.
Since it is a reduction of the set of admissible minimizers, any upper bound derived for this simpler model will also hold for \eqref{def:EnonlocDimensional}, so at the level of upper bounds, our results are unaffected by this constraint. 
Moreover, numerical simulations suggest that this change has little qualitative effect on the behavior of the system \cite{Jagla07}.

One of the major difficulties with \eqref{def:EnonlocDimensional} is the complex, nonlocal dependence
of the substrate energy on the geometry of the set $\Omega$, which defines the boundary conditions for the linear elasticity problem in the substrate. 
Rather than removing potential pathologies by imposing artificial constraints on $\Omega$, we choose to study an approximate energy which is local but has the same scaling properties.  
That is, we work with the following approximation: 
\begin{align}
E_3[u,w,\Omega] & = \frac{\alpha_m t}{L^2}\int_{[0,L]^2} \abs{e(w) + \frac{1}{2}\grad u \otimes \grad u - \eta I}^2 dx + \frac{t^3}{L^2}\int_{[0,L]^2} \abs{D^2u}^2 dx \nonumber \\ & \quad + \frac{\alpha_s}{L^2} \left(\int_{\Omega} \abs{\grad w}^2 dx \right)^{1/2} \left( \int_{\Omega} \abs{w}^2 dx \right)^{1/2},  \label{def:ElocalDimen}
\end{align}
constrained by the condition $u = 0$ in $\Omega$.  
The motivation for our local approximation of the substrate energy is the scale-invariant interpolation inequality $\norm{f}^{2}_{\dot{H}^{1/2}} \lesssim \norm{f}_{2}\norm{\grad f}_2$.
Due to the scale invariance, the scaling laws for \eqref{def:ElocalDimen} and \eqref{def:EnonlocDimensional} are expected to agree up to constants in most physical situations. 
Indeed, the upper bound estimates we present are equally valid for \eqref{def:ElocalDimen} and \eqref{def:EnonlocDimensional} as can be verified by a simple scaling argument: the upper bounds are proved by constructing explicit test functions of the form 
\begin{equation*}
\left(w_l(x,y),u_l(x,y)\right) = l \left(w_1\left( \frac{x}{l}, \frac{y}{l} \right), u_1 \left(\frac{x}{l}, \frac{y}{l}\right) \right), 
\end{equation*} 
where $l$ is a characteristic length-scale chosen in terms of the parameters of the problem. 
By the scale-invariance, it follows that the upper bound implied by the construction will be the same for both \eqref{def:ElocalDimen} and \eqref{def:EnonlocDimensional}. 

Finally, we non-dimensionalize length with respect to $L$. 
Choosing the scaling
\begin{align*}
w(x,y) = L \tilde{w}\left(\frac{x}{L},\frac{y}{L}\right), \quad u(x,y) = L\tilde{u}\left(\frac{x}{L},\frac{y}{L}\right), \quad \tilde \Omega = \set{\tilde x \in \Torus^2: L\tilde x \in \Omega},  
\end{align*}     
we may write the energy as follows (denoting $\tilde{x} = x/L$), 
\begin{align} 
E_3[u,w,\Omega] & = L\left[\alpha_m \frac{t}{L}\int_{[0,1]^2} \abs{e(\tilde{w}) + \frac{1}{2}\grad \tilde{u} \otimes \grad \tilde{u} - \eta I}^2 d\tilde{x} + \frac{t^3}{L^3}\int_{[0,1]^2} \abs{D^2 \tilde u}^2 d\tilde{x}\right] \nonumber \\ & \quad + L\left[\alpha_s \left(\int_{[0,1]^2 \cap \tilde \Omega} \abs{\grad \tilde w}^2 d\tilde{x} \right)^{1/2} \left( \int_{[0,1]^2 \cap \tilde \Omega} \abs{\tilde w}^2 d\tilde{x}\right)^{1/2} \right] \nonumber \\ 
& = L E_{ND}[\tilde{u},\tilde{w},\tilde \Omega]. \label{eq:dimensional}  
\end{align} 
Dropping the tildes and setting $h = t/L$ (the non-dimensionalized film thickness) we have reduced our problem to the study of the non-dimensionalized energy
\begin{align}
E_{ND}[u,w,\Omega] & = \alpha_m h\int_{[0,1]^2} \abs{e(w) + \frac{1}{2}\grad u \otimes \grad u - \eta I}^2 dx + h^3\int_{[0,1]^2} \abs{D^2u}^2 dx \nonumber \\ & \quad + \alpha_s \left(\int_{\Omega} \abs{\grad w}^2 dx \right)^{1/2} \left( \int_{\Omega} \abs{w}^2 dx \right)^{1/2}.  \label{def:Elocal}
\end{align}
By \eqref{eq:dimensional}, any statement about \eqref{def:Elocal} can be translated to a statement about \eqref{def:ElocalDimen}.
For the remainder of the paper we concentrate on \eqref{def:Elocal} with no further comment.     

To summarize, given $\theta \in (0,1)$, $\alpha_s,h,\eta > 0$, we are interested in determining upper and lower bounds on the following variational problem:
\begin{align}
E_{h,\alpha_s,\eta,\theta} & := \min_{(w,u,\Omega) \in \mathcal A} E_{ND}[w,u,\Omega], \label{def:VarP}
\end{align}
with the admissible class of minimizers given by 
\begin{align}   
\mathcal{A} & := \set{(w,u,\Omega)| w \in H^1(\Torus^2;\Real^2), u \in H^2(\Torus^2;[0,\infty)), \Omega \subset \Torus^2 \textup{ closed }, \abs{\Omega} = \theta, u|_{\Omega} = 0}, \label{def:A} 
\end{align}
with the obvious modifications for the 1D case.

\subsection{Statement of Results} \label{sec:Statement-of-Results}
In one dimension we can identify the optimal scaling law for \eqref{def:VarP} with respect to $h$,$\eta$ and $\alpha_s$; that is, we identify the minimum energy up to a prefactor which depends only on $\theta$.  
The upper bound in \eqref{ineq:upperbound1d2} is proved by considering a test function
which is a periodic arrangement of blisters and bonded regions. 

\begin{theorem} \label{thm:oneDLoc}
In one dimension, \eqref{def:VarP} satisfies
\begin{itemize}
\item[(i)] the lower bound 
\begin{align}
E_{h,\alpha_s,\eta,\theta} \geq K_1\min\left(\alpha_m \eta^{2}\theta^2, \alpha_s^{2/3} \eta^{5/3} \frac{\theta^{5/3}}{(1-\theta)^{1/3}}\right)h, \label{ineq:lowerbd1d}
\end{align} 
for some constant $K_1 > 0$; 
\item[(ii)] the upper bound 
\begin{align} 
\min\left(\alpha_m \eta^2, \alpha_m \eta^2 \theta + K_2\frac{h^2 \eta}{1-\theta} \right)h \geq E_{h,\alpha_s,\eta,\theta} \label{ineq:upperbound1d1}
\end{align} 
for some constant $K_2 > 0$.
The first upper bound is exhibited by $u = w = 0$ and the second is exhibited by a single large blister, setting $u = w = 0$ in the bonded region. 
\item[(iii)] If we additionally have 
\begin{align}
l_1 := \frac{h}{\eta^{1/3} \alpha_s^{1/3}} < c_0(1-\theta)^{2/3}\theta^{2/3}, \label{cond:1D}
\end{align}
for some constant $c_0$ then \eqref{def:VarP} also satisfies the upper bound 
\begin{align}
K_3\frac{\theta^{4/3}}{(1-\theta)^{2/3}}\alpha_s^{2/3}\eta^{5/3}h \geq E_{h,\alpha_s,\eta,\theta},  \label{ineq:upperbound1d2}
\end{align}
for some constant $K_3 > 0$.
This upper bound is exhibited by a periodic pattern of blisters and bonded regions with characteristic length-scale $\sim l_1$.  
\end{itemize}
\end{theorem}

\begin{remark}\textbf{Physical interpretation:} Assuming for simplicity $\theta \sim \frac{1}{2}$, the physical meaning of the theorem is as follows: when $\eta \lesssim \alpha_s^{2}$, then either no blistering or one large blister is expected, whereas if $\eta \gtrsim \alpha_s^{2}$ then a periodic array of blisters is expected. Condition \eqref{cond:1D} ensures that at least one period fits inside the unit interval and should be seen as a smallness condition on $h$. See Remark \ref{rmk:2dphase} below for a brief discussion of the extreme values of $\theta$. 
\end{remark} 

\begin{remark}\textbf{Characteristic length scale:} The optimal length scale $l_1$ defined by \eqref{cond:1D} is determined by the competition between bending energy in the blisters and elastic energy in the substrate. 
\end{remark}

In two dimensions, our upper and lower bounds do not agree in their dependence on the parameters of $\alpha_s$ and $\eta$. 
The presence of $\alpha_m$ in \eqref{ineq:upperbound2d2} below indicates that the membrane energy is playing a leading order role; this is the reason the upper and lower bounds do not agree. 
Our upper bound is thus quite different from that known for the totally bonded case in the wrinkled regime, where the membrane energy is higher order and does not enter into the optimal scaling law \cite{KohnNguyen12}.
Since our upper and lower bounds scale differently with respect to $\eta$ and $\alpha_s$, at least one must be suboptimal, however doing better seems to require a new idea. 
We discuss in Section \ref{sec:2DUpper} the relationship between the 2D blistering problem and the problem of confining thin elastic sheets, also called the crumpling problem \cite{ContiMaggi08}; but we remark here that it is not currently known how to obtain matching upper and lower bounds on most problems of this type. 
The upper bound \eqref{ineq:upperbound2d2} is proved using a two-dimensional lattice of blisters.  
The construction involves the use of `minimal ridges', studied for example in \cite{Lobkovsky96,LobkovskyWitten97,Venkataramani03,AudolyPomeau}; in particular we use a F\"oppl-von K\'arm\'an analogue of the fully nonlinear construction of Conti and Maggi \cite{ContiMaggi08}.

\begin{theorem}[Two Dimensions] \label{thm:2d}
In two dimensions, \eqref{def:VarP} satisfies
\begin{itemize} 
\item[(i)] the lower bound, 
\begin{align} 
E_{h,\alpha_s,\eta,\theta} \geq K_4\min\left(\alpha_m \eta^2 \theta^3, \alpha_s^{2/3}\eta^{5/3} \theta^{8/3}  \right) h, \label{ineq:lowerbound2d}
\end{align}
for some constant $K_4 > 0$. 
\item[(ii)] the upper bound  
\begin{align} 
\min\left(\alpha_m \eta^2, \alpha_m \eta^2 \theta + K_5(\theta) \alpha_m \eta^{3/2}h \right) h \geq E_{h,\alpha_s,\eta,\theta}, \label{ineq:upperbound2d1}
\end{align}
where $K_5(\theta) > 0$ depends only on $\theta$.
The first upper bound is exhibited by $u = w = 0$. The second is exhibited by the constructions of \cite{JinSternberg01,BelgacemContiSimoneMuller00}, in which a single large blister exhibiting fine-scale wrinkles near the blister edge is chosen with $u = w = 0$ in the bonded region. 
\item[(iii)] If we have
\begin{align}
l_2 := \frac{\alpha_m^{1/16} h}{\eta^{5/16} \alpha_s^{3/8}}  & < c_2(\theta), \label{cond:2d1a}  \\ 
\frac{h}{\sqrt{\alpha_m \eta}} & < c_1(\theta)l_2 \label{cond:2d1b} 
\end{align}
and
\begin{align}
\eta < c_3(\theta)\frac{\alpha_s^{2/17}}{\alpha_m^{3/17}} \label{cond:2d2}
\end{align}
where $c_1,c_2,c_3 > 0$ depend only on $\theta$, then \eqref{def:VarP} also satisfies the upper bound 
\begin{align} 
K_6(\theta)\alpha_m^{1/16} \alpha_s^{5/8} \eta^{27/16} h \geq E_{h,\alpha_s,\eta,\theta},  \label{ineq:upperbound2d2}
\end{align}
where $K_6 > 0$ depends only on $\theta$. 
This upper bound is exhibited by a periodic lattice of blisters and bonded regions with characteristic length-scale $\sim l_2$.  
\end{itemize}
\end{theorem} 
\begin{remark} \label{rmk:2dphase} \textbf{Physical interpretation:} The physical meaning of the theorem can be explained as follows, first assuming $\theta \sim 1/2$. 
If $\eta \lesssim \alpha_s^2$ then the upper bound \eqref{ineq:upperbound2d1} is optimal. Moreover, this is precisely the case where \eqref{cond:2d1b} fails. Hence in the region $\eta \lesssim \alpha_s^2$, either one large blister or no blisters at all are expected (or at least, achieve the optimal scaling law). 
If $\alpha_s^2 \lesssim \eta \lesssim \alpha_s^{2/17}$ then, if $h$ is sufficiently small (the condition \eqref{cond:2d1a}), the upper bound in \eqref{ineq:upperbound2d2} improves on \eqref{ineq:upperbound2d1} and a blister web is expected. 
If $\eta \gtrsim \alpha_s^{2/17}$ then \eqref{cond:2d2} fails. 
In this case, the terms dropped from the nonlinear model to obtain the F\"oppl-von K\'arm\'an energy are not negligible in our minimal ridge construction (Lemma \ref{lem:MinRidge}). 
Hence, physically speaking \eqref{cond:2d2} can be likened to a modeling requirement that ensures that the F\"oppl-von K\'arm\'an model can be used.  
The condition \eqref{cond:2d1a} is necessary to ensure that the lattice fits into the unit periodic square (analogous to \eqref{cond:1D}) whereas \eqref{cond:2d1b} is necessary to ensure that the minimal ridges which obtain the optimal scaling fit within the geometric confinements of the construction. 
See Figure \ref{fig:Phase2d} for a graphical depiction of the situation. 

We have not carefully quantified all of the constants with respect to the limits $\theta \rightarrow 0$ and $\theta \rightarrow 1$, however, we may still make some comments. 
In the limit $\theta \rightarrow 0$, the blistering problem reduces to a type of crumpling problem such as that studied in \cite{ContiMaggi08}. 
Indeed, despite being totally debonded in that limit, the film is still constrained by the periodic boundary conditions and is expected to crumple in a manner similar to that considered in \cite{ContiMaggi08}, for example with a Miura-ori pattern.
Since these problems are widely expected to scale with energy $O(h^{8/3})$ we expect that for sufficiently small $\theta$ the film may not form blister webs.
If $\theta \rightarrow 1$ then the film should be expected to have no incentive to appreciably debond but the above theorem does not make this precise. \begin{center}
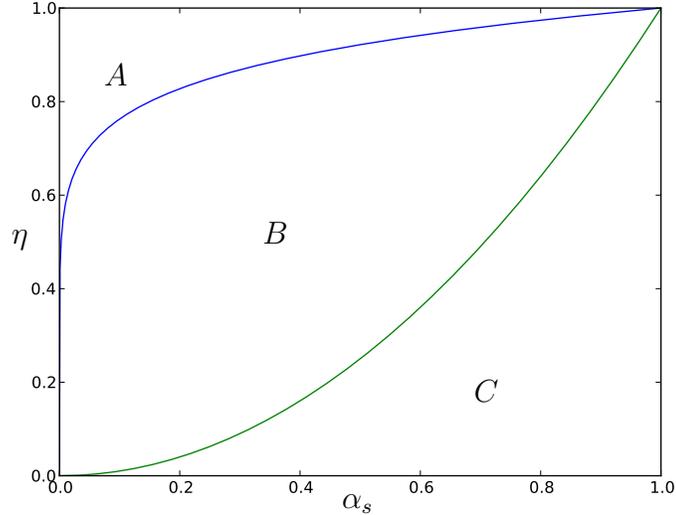
\begin{figure}[hbpt] 
\begin{picture}(150,200)(-100,0)
\scalebox{.5}{\includegraphics{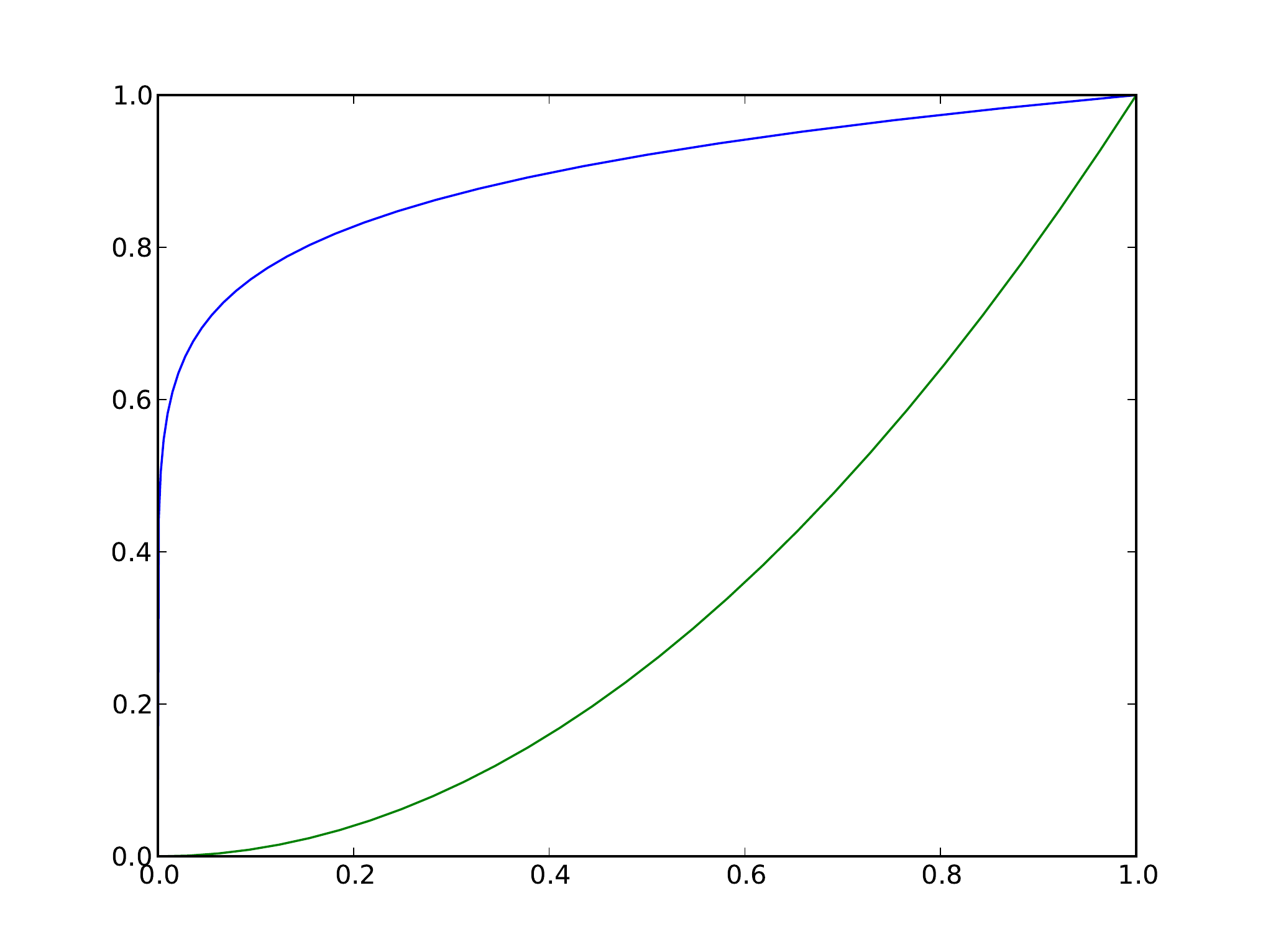}}
\put(-275,110){$\eta$}
\put(-150,10){$\alpha_s$}
\put(-240,170){$A$}
\put(-100,50){$C$}
\put(-180,110){$B$}
\end{picture}
\caption{A graphical depiction of the situation described by Theorem \ref{thm:2d} in the case $\theta \sim \frac{1}{2}$. In the region (C), Theorem \ref{thm:2d} (i) and (ii) show that no blistering or a single large blister achieves the optimal scaling law. In the region (B) between the two curves, Theorem \ref{thm:2d} (iii) shows that a lattice-like arrangement of blisters can attain lower energies than large isolated blisters (but does not show that this arrangement achieves the optimal scaling law in this regime).
In region (A), the construction of (iii) no longer obtains the scaling \eqref{ineq:upperbound2d2} in the F\"oppl-von K\'arm\'an setting, and a fully nonlinear model might be more appropriate. } \label{fig:Phase2d}
\end{figure}
\end{center} \end{remark} 
\begin{remark}\textbf{Characteristic length scale:} The optimal length scale $l_2$ set in \eqref{cond:2d1a} is determined by the competition between the elastic energy in the substrate and the energy of a crumpling-type problem in the blisters, which itself involves both the bending and membrane energy. We also remark that the majority of the crumpling energy arises where the horizontal and vertical components of the blister lattice cross, as it is here that the two dimensional nature of the problem becomes crucial; see Section \ref{sec:2DUpper}.  
\end{remark} 

\begin{remark}\textbf{Experiments:} 
In the 1D case, periodic arrangements very similar to that used to prove \eqref{ineq:upperbound1d2} have been observed in experiments of one-dimensional blistering \cite{VellaEtAl09}.
The the characteristic length scale of these periodic patterns and of 2D blister webs are observed in experiments to depend linearly on $h$, consistent with our work (see e.g. \cite{MoonJensenEtAl02,YuZhangChen09}). 
\end{remark}

\begin{remark} \label{rmk:CompBond} \textbf{Comparison with bonded case:} It is useful to compare Theorem \ref{thm:2d} with what is already known.
Consider first the case of a totally bonded film, $\theta = 1$. 
In \cite{AudolyBoudaoudIII08}, the authors discuss the so-called Miura-ori pattern which exhibits an energy scaling law of $E_{ND} = O(\alpha_m^{1/16}\alpha_s^{5/8}\eta^{17/16})$, larger than \eqref{ineq:upperbound2d2} just by a power of $\eta$ (similar to that seen in the 1D calculation sketched in Section \ref{sec:FvK}). 
The test function we employ to prove \eqref{ineq:upperbound2d2} involves a ridge construction like that used in \cite{AudolyBoudaoudIII08}, however the energy arising from the substrate is lower (by a power of $\eta$) than the energy in the completely bonded case, since our test function does not buckle out-of-plane in the bonded regions.  
Since $\eta$ is the small mismatch strain, this confirms the intuition that debonding allows the film to sustain lateral compression with an energy much smaller than that of a bonded film.
\end{remark} 

\begin{remark}\textbf{Solitary blisters: } The second possibility in the upper bound \eqref{ineq:upperbound2d1} arises from the construction of \cite{JinSternberg01,BelgacemContiSimoneMuller00}. 
It involves a fixed $\Omega$ (independent of the small parameters) with no deformation in the bonded region (so there is no substrate energy). 
As $h \rightarrow 0$ the blister develops fine wrinkles along its boundary due to the inconsistency between the membrane energy and the compression at the edge of the blister.
Such behavior is expected in the case of a relatively stiff substrate. 
\end{remark} 

\section{One Dimension} 
\subsection{Upper Bound}\label{sec:UppBd1D}
\begin{proof}(Theorem \ref{thm:oneDLoc} (ii),(iii))
We first prove the upper bounds stated in (ii) and (iii) as these are relatively straightforward and the upper bound construction of (iii) will provide helpful intuition for the proof of (i). 
For the first upper bound in \eqref{ineq:upperbound1d1}, it suffices to consider the trivial test function $u = w = 0$. 
The second upper bound in \eqref{ineq:upperbound1d1} follows by choosing $u = w = 0$ in an interval of length $\theta$ and choosing a smooth blister with zero membrane energy in the remaining interval of length $1-\theta$. 
To be precise, it suffices to take 
\begin{align*}
w(x) & = \left\{
\begin{array}{lr}
\frac{\eta(1-\theta)}{4\pi}\sin\left(\frac{4\pi x}{1-\theta}\right) & 0 \leq x \leq 1-\theta \\ 
0 & 1-\theta < x \leq 1
\end{array}
,\right. \\ 
u(x) & = \left\{
\begin{array}{lr}
\frac{2\sqrt{\eta}(1-\theta)}{\pi}\sin^2\left(\frac{\pi x}{1-\theta}\right) & 0 \leq x \leq (1-\theta) \\ 
0 & 1-\theta < x \leq 1
\end{array} 
.\right. 
\end{align*}

Now we turn to the upper bound in \eqref{ineq:upperbound1d2}.   
Divide the interval into $l^{-1}$ regions of length $l$, for a scale to be chosen later.  
On intervals of length $\theta l$ take $u = 0$ and $w = \eta x - x_i$ where $x_i$ is the center of the interval. 
On the intervals of length $(1-\theta)l$ left over, we place a blister which achieves zero membrane energy and finite bending energy. 
There is a large amount freedom in choosing suitable configurations (a fact which we exploit later in Lemma \ref{lem:MinRidge}), but the present situation is so simple we can make an explicit choice: 
\begin{align*} 
w_1(x) & = \left\{ 
\begin{array}{lr}
\eta\left(1 - \frac{1}{1-\theta}\right)x + \frac{\theta \eta}{2} + \frac{\eta}{4\pi}\sin\left(\frac{4\pi x}{1-\theta}\right)  & 0 \leq x \leq 1-\theta \\ 
\eta \left(x - \frac{2-\theta}{2}\right) & 1-\theta < x \leq 1
\end{array}
,\right. \\  
u_1(x) & = \left\{ 
\begin{array}{lr}
2\sqrt{\frac{\eta}{1-\theta}}\left(\frac{1-\theta}{2\pi}\right)\left[1 - \cos\left(\frac{2\pi x}{1-\theta}\right)\right] & 0 \leq x \leq 1-\theta \\ 
0 & 1-\theta < x \leq 1
\end{array} 
,\right. \\ 
\Omega_1 &= [1-\theta,1], \\  
w(x) & = lw_1\left(\frac{x}{l}\right), \\  
u(x) & = lu_1\left(\frac{x}{l}\right), \\ 
\Omega_l & = [(1-\theta)l,l], 
\end{align*} 
with $w,u$ and $\Omega_l$ extended periodically over $[0,1]$.

See figure \ref{fig:Periodic1D} for a rough depiction of this function. 
\begin{center}
\begin{figure}[hbpt] 
\begin{picture}(50,50)(-25,0)
\put(125,5){$\theta l$}
\put(190,40){$(1-\theta)l$}
\scalebox{1.00}{\includegraphics{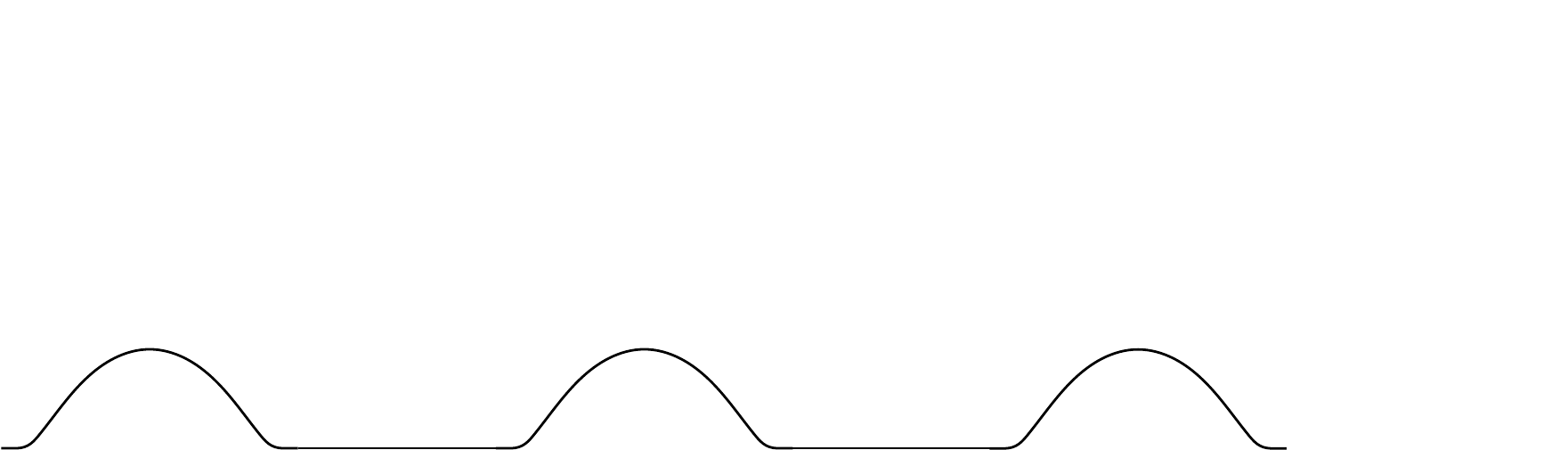}}
\end{picture}
\caption{A rough depiction of the out-of-plane displacement in the 1D periodic construction. The bonded regions are of length $\theta l$ while the blistered regions are of length $(1-\theta)l$.}
\label{fig:Periodic1D}
\end{figure}
\end{center}
One can check that the bending energy satisfies
\begin{equation*}
\int_0^{(1-\theta)l} \abs{u_{xx}}^2 dx \sim \frac{\eta}{(1-\theta)^2l}. 
\end{equation*}
Moreover, the localized substrate energy satisfies 
\begin{equation*}
\alpha_s \left( \int_{(1-\theta)l}^l \abs{w_x}^2 dx \right)^{1/2} \left( \int_{(1-\theta)l}^l \abs{w}^2 dx \right)^{1/2} \sim \alpha_s\eta^2 \theta^{2} l^2.
\end{equation*} 
Using $\sum a_jb_j = \left(\sum a_j^2\right)^{1/2}\left(\sum b_j^2\right)^{1/2}$ when $a_j \equiv a$ and $b_j \equiv b$ for all $j$, 
it follows since every bonded interval is the same that (also using that there are $l^{-1}$ intervals), 
\begin{align*} 
\alpha_s \left( \int_{\Omega} \abs{w_x}^2 dx \right)^{1/2} \left( \int_{\Omega} \abs{w}^2 dx \right)^{1/2} \sim \alpha_s\eta^2 \theta^{2} l.
\end{align*} 
Hence the total energy of the construction is estimated by
\begin{equation*}
E \sim h^3\frac{\eta}{(1-\theta)^2l^2} + \alpha_s \eta^2 \theta^2 l. 
\end{equation*}
Optimizing in $l$ implies
\begin{align*}
l \sim \frac{h}{\eta^{1/3} \alpha_s^{1/3}(1-\theta)^{2/3} \theta^{2/3}}
\end{align*}
and 
\begin{equation*}
E_{ND} \sim h \eta^{5/3} \alpha_s^{2/3} \frac{\theta^{4/3}}{(1-\theta)^{2/3}}. 
\end{equation*}
The condition \eqref{cond:1D} is simply that the predicted characteristic length-scale $l$ is smaller than the entire periodic interval of length one. Hence (iii) follows. 
\end{proof}   

\subsection{Lower Bound}
\begin{proof}(Theorem \ref{thm:oneDLoc} (i))
We now turn to the lower bound \eqref{ineq:lowerbd1d}, which is the most difficult statement in Theorem \ref{thm:oneDLoc}, mainly due to the ambiguity in the geometry of the bonded region.
To reduce clutter of notation, in the following proof we regularly refer to intervals and their length interchangeably. 
Consider an arbitrary deformation $(w,u) \in \mathcal{A}$ (defined in \eqref{def:A}).  
Notice that since $\int \abs{u_{xx}}^2 dx < \infty$, $u$ is continuously differentiable and since $\int \abs{w_x}^2 dx < \infty$, $w$ is continuous.   
We refer to the open debonded intervals as `blisters'. 
Label a blister of length $l_i$ as `good' if $\max_{x \in l_i} \abs{u_x} > \sqrt{\eta\theta}/K$, 
where
\begin{align} 
K = \sqrt{8\sqrt{12}}. \label{def:K}
\end{align} 
The reason for this choice of $K$ is technical and will be apparent later in the proof.  
Since, for $x \in l_i$, 
\begin{align*} 
\abs{u_x(x)} \leq \int_{l_i} \abs{u_{xx}(y)} dy \leq l_i^{1/2} \left(\int_{l_i} \abs{u_{xx}(y)}^2 dy\right)^{1/2}, 
\end{align*}
it follows that the bending energy in good blisters is bounded below by $\eta h^{3} \theta K^{-2} l_i^{-1}$. 
Hence, although there may be infinitely many blisters, there can only be finitely many good blisters, as otherwise the total energy would be infinite. 
Now consider the $N$ interim regions of length $L_i$ which separate the good blisters. On these intervals, the energy contribution is bounded below by
\begin{align*}
E_i = E_{ND}[u|_{L_i},w|_{L_i}] \geq \alpha_mh \int_{L_i} \abs{w_x + \frac{1}{2}\partial_x u^2 - \eta}^2 dx + \alpha_s \norm{w_x}_{L^2(\Omega \cap L_i)}\norm{w}_{L^2(\Omega \cap L_i)}. 
\end{align*} 
Denote by $d_i$ the total length of the bonded region in the $i-$th interim region. 
It is not necessarily the case that $d_i = L_i$; indeed there may be infinitely many blisters in the interim region. 
See Figure \ref{fig:NonPeriodic1D} for a diagram of this decomposition, and compare against Figure \ref{fig:Periodic1D}. 

\begin{center}
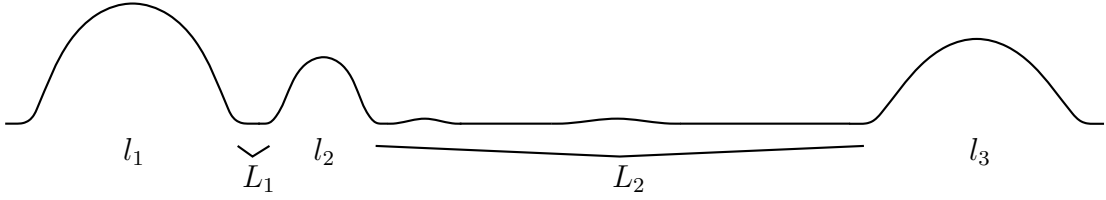
\begin{figure}[hbpt] 
\begin{picture}(50,75)(-25,0)
\put(45,15){$l_1$}
\put(90,5){$L_1$}
\put(117,15){$l_2$}
\put(230,5){$L_2$}
\put(365,15){$l_3$}
\scalebox{1.00}{\includegraphics{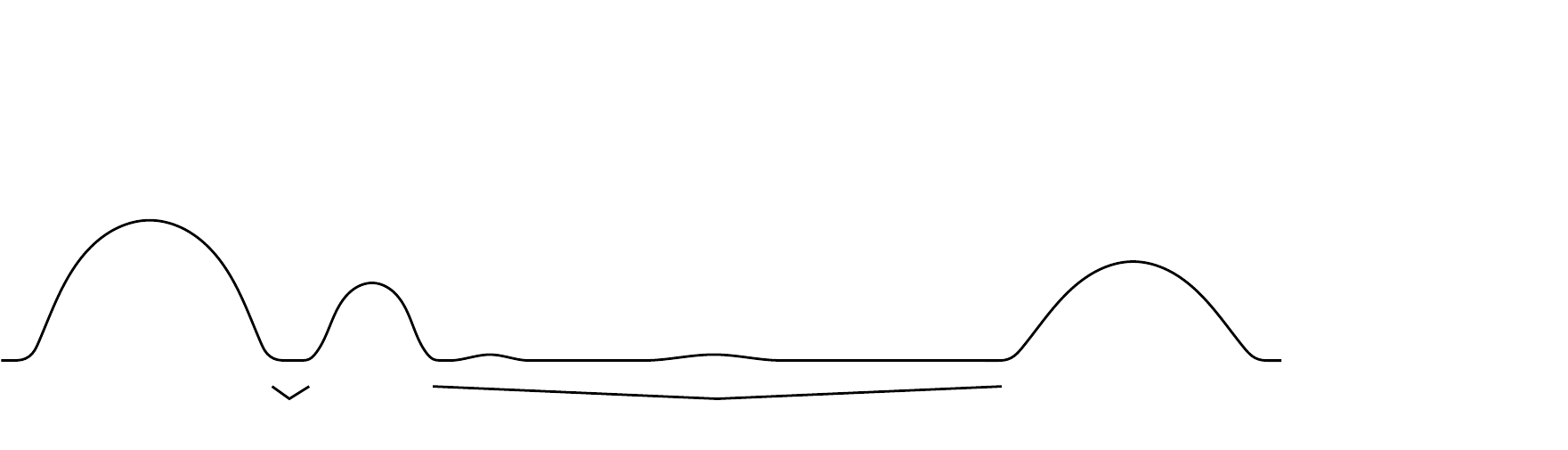}}
\end{picture}
\caption{A depiction of a typical out-of-plane displacement in the 1D lower bound. Notice that $L_2$, the interim region between good blisters $l_2$ and $l_3$, is not totally bonded. Compare against Figure \ref{fig:Periodic1D}.}
\label{fig:NonPeriodic1D}
\end{figure}
\end{center}

Let $M_i >0$ be fixed later (depending on $i$) to be the following, for reasons which will be apparent later: 
\begin{align} 
\sqrt{M}_i = \frac{\eta d_i}{2\sqrt{12 L_i}}. \label{ineq:Mi2} 
\end{align} 
We will divide into two cases. 
The first is when at least $\theta/2$ of the bonded region is contained in 
interim regions such that $\int_{L_i} \abs{w_x + \frac{1}{2}u_x^2 - \eta}^2 dx < M_i$ and the second is when at least $\theta/2$
of the bonded region is contained in interim regions such that $\int_{L_i} \abs{w_x + \frac{1}{2}u_x^2 - \eta}^2 dx \geq M_i$.

{\it Case 1: $\theta/2$ of bonded region in interim regions such that $\int_{L_i} \abs{w_x + \frac{1}{2}u_x^2 - \eta}^2 dx < M_i$:} \\
This is the more difficult case, and roughly corresponds to a large fraction of the bonded region having low membrane energy (but is not exactly equivalent). 
It further follows that $d_i \geq \theta L_i/4$ in a set of interim regions which contains at least $\theta/4$ worth of the bonded region we are considering (as otherwise the length of the interim regions would total more than one). 
Consider now only these $M \leq N$ interim regions in which $\int_{L_i} \abs{w_x + \frac{1}{2}u_x^2 - \eta}^2 dx < M_i$ and $d_i \geq \theta L_i/4$ and ignore the others which we cannot control.  
Let $c_{i,h}$ be such that $w-\eta x - c_{i,h}$ is average zero over the $i$-th interim region and write $F_i(x) = \eta x + c_{i,h}$. 
Then, 
\begin{align*} 
\alpha_s\norm{w_x}_{L^2(\Omega \cap L_i)}\norm{w}_{L^2(\Omega \cap L_i)} & \geq \alpha_s\norm{\partial_xF_i}_{L^2(\Omega \cap L_i)}\norm{w}_{L^2(\Omega \cap L_i)} - \alpha_s\norm{w_x - \partial_xF_i}_{L^2(\Omega \cap L_i)}\norm{w}_{L^2(\Omega \cap L_i)} \\ 
& \geq \alpha_s\norm{w}_{L^2(\Omega \cap L_i)} \left(\eta d_i^{1/2} - \norm{w_x - \partial_xF_i}_{L^2(\Omega \cap L_i)} \right). 
\end{align*}
Since 
\begin{align*} 
\norm{w_x - \partial_x F_i}_{L^2(\Omega \cap L_i)} = \norm{w_x - \eta}_{L^2(\Omega \cap L_i)} \leq \sqrt{M_i},
\end{align*}
and \eqref{ineq:Mi2} implies since $d_i \leq L_i$,  
\begin{equation}
\sqrt{M}_i < \frac{1}{2}\eta d_i^{1/2},
\end{equation}
we have
\begin{align*} 
\alpha_s\norm{w_x}_{L^2(\Omega \cap L_i)}\norm{w}_{L^2(\Omega \cap L_i)}  & \geq \frac{\alpha_s}{2}\norm{w}_{L^2(\Omega \cap L_i)} \eta d_i^{1/2} \\ 
& \geq \frac{\alpha_s}{2}\eta d_i^{1/2}\left(\norm{F_i}_{L^2(\Omega \cap L_i)} - \norm{w - F_i}_{L^2(\Omega \cap L_i)} \right) \\ 
& \geq \frac{\alpha_s}{2} \eta d_i^{1/2}\left( \frac{1}{\sqrt{12}} \eta d_i^{3/2} - \norm{w - F_i}_{L^2(\Omega \cap L_i)} \right). 
\end{align*}
Since $w -F_i$ is absolutely continuous and mean-zero over $L_i$, 
\begin{align*} 
\norm{w-F_i}_{L^2(\Omega \cap L_i)} & \leq \sqrt{d_i} \max_{\Omega \cap L_i}\abs{w - F_i} \\ 
 & \leq \sqrt{d_i} \int_{L_i} \abs{w_{x} - \eta} dx \\ 
& \leq \sqrt{d_i L_i} \norm{w_x - \eta}_{L^2(L_i)}. 
\end{align*} 
By the triangle inequality and our assumptions on the membrane energy and the definition of good blister, 
\begin{align*} 
 \norm{w_x - \eta}_{L^2(L_i)} \leq \sqrt{M_i} + \frac{1}{2}\norm{u_x^2}_{L^2(L_i)} \leq \sqrt{M_i} + \frac{\eta \theta}{2K^2}\sqrt{L_i}. 
\end{align*} 
Therefore we have, 
\begin{align*} 
\alpha_s\norm{w_x}_{L^2(\Omega \cap L_i)}\norm{w}_{L^2(\Omega \cap L_i)} & \geq \frac{\alpha_s}{2}\eta d_i^{1/2}\left(\frac{1}{\sqrt{12}}\eta d_i^{3/2} - \sqrt{d_i L_i}\sqrt{M_i} - \frac{\eta \theta}{2K^2}\sqrt{d_i} L_i \right).  
\end{align*} 
By \eqref{ineq:Mi2}, 
\begin{align*} 
\alpha_s\norm{w_x}_{L^2(\Omega \cap L_i)}\norm{w}_{L^2(\Omega \cap L_i)} & \geq \frac{\alpha_s}{2}\eta d_i^{1/2}\left(\frac{1}{2\sqrt{12}} \eta d_i^{3/2} - \frac{\eta \theta}{2K^2}\sqrt{d_i} L_i\right).  
\end{align*} 
Since $d_i \geq L_i \theta/4$, on the interim regions we are considering, it follows that (and recalling our choice of $K$ \eqref{def:K}), 
\begin{align*} 
\alpha_s\norm{w_x}_{L^2(\Omega \cap L_i)}\norm{w}_{L^2(\Omega \cap L_i)} & \geq \frac{\alpha_s}{2}\eta d_i^{1/2}\left(\frac{1}{2\sqrt{12}} \eta d_i^{3/2} - \frac{2\eta}{K^2}d_i^{3/2}\right) \\ 
& \geq \frac{\alpha_s}{8\sqrt{12}}\eta^2 d_i^{2} = \frac{\alpha_s}{K^2}\eta^2 d_i^{2}.    
\end{align*}
Summing over all interim regions that satisfy our assumptions and all good blisters, the total energy is bounded below by (labeling the interim intervals of low membrane energy and large bonded region $I(i)$),  
\begin{align*} 
E_{ND}[w,u] \geq \frac{h^3 \eta \theta}{K^2}\sum_{i = 1}^N \frac{1}{l_i} + \frac{\eta^2\alpha_s}{K^2} \sum_{i = 1}^M  d_{I(i)}^{2}. 
\end{align*}
By our choice of interim regions and Cauchy-Schwarz, 
\begin{equation*}
\frac{\theta}{4} \leq \sum_{i = 1}^M d_{I(i)} \leq M^{1/2}\left(\sum_{i = 1}^M d_{I(i)}^{2} \right)^{1/2} \leq N^{1/2}\left(\sum_{i = 1}^M d_{I(i)}^{2} \right)^{1/2}, 
\end{equation*}
and 
\begin{equation*}
N^2 = \left(\sum_{i = 1}^N \frac{\sqrt{l_i}}{\sqrt{l_i}}\right)^2 \leq \left(\sum_{i = 1}^N l_i \right)\left(\sum_{i = 1}^N \frac{1}{l_i}\right) = (1-\theta)\sum_{i = 1}^N \frac{1}{l_i},  
\end{equation*}
which implies that the total energy is bounded below by
\begin{equation*}
E_{ND}[w,u] \geq h^3 \eta \frac{\theta N^2}{(1-\theta)K^2} + \frac{\alpha_s}{16 K^2} \eta^2 \theta^{2}N^{-1}. 
\end{equation*}
Optimizing in $N$ implies the energy estimate
\begin{equation*}
E_{ND}[w,u] \gtrsim \eta^{5/3} \alpha_s^{2/3} \frac{\theta^{5/3}}{(1-\theta)^{1/3}} h.
\end{equation*}
This proves the energy estimate in the first case. 

{\it Case 2: $\theta/2$ of bonded region in interim regions such that $\int_{L_i} \abs{w_x + \frac{1}{2}u_x^2 - \eta}^2 dx \geq M_i$:} \\
This case is easier and roughly corresponds to a large fraction of the bonded region having high membrane energy (but as above is not exactly equivalent).  
Arguing as we did in case 1, $d_{i}L_{i}^{-1} > \theta/4$ on $M \leq N$ of these interim intervals such that they include at least $\theta/4$ worth of the bonded region. 
Then, summing the membrane energy over these $M$ intervals, labeled $I(i)$, 
\begin{equation*}
E[w,u] \geq \alpha_m h\sum_{i = 1}^M M_{I(i)} = \alpha_m h \sum_{i = 1}^M \frac{\eta^2 d_{I(i)}^2}{48 L_{I(i)}} > \frac{\alpha_m h \eta^2 \theta}{192} \sum_{i = 1}^M d_{I(i)} > \frac{\alpha_m h \eta^2 \theta^2}{768}.
\end{equation*}
This completes the proof of the one dimensional lower bound \eqref{ineq:lowerbd1d}, concluding the proof of Theorem \ref{thm:oneDLoc}. 
\end{proof}

\section{Two Dimensions} \label{sec:2D}
\subsection{Lower Bound} 
\begin{proof}(Theorem \ref{thm:2d} (i))
The lower bound \eqref{ineq:lowerbound2d} follows easily from the one dimensional estimate \eqref{ineq:lowerbd1d}.
Define the projection onto the $x$ axis defined by $\pi(x) = (x_1,0)$ where $x = (x_1,x_2)$. 
We are assuming the bonded region $\Omega$ is closed and therefore $\pi^{-1}(x) \cap \Omega$ is measurable.  
By Fubini, 
\begin{align}
\theta  = \int_0^1 \abs{ \pi^{-1}(x) \cap \Omega} dx & = \int_{\set{x \in \Torus^1 : \abs{\pi^{-1}(x) \cap \Omega} > \theta/2}} \abs{ \pi^{-1}(x) \cap \Omega} dx \\ & \quad + \int_{\set{x \in \Torus^1 : \abs{\pi^{-1}(x) \cap \Omega} \leq \theta/2}} \abs{ \pi^{-1}(x) \cap \Omega} dx  \nonumber  \\
& \leq \int_{\set{x \in \Torus^1 : \abs{\pi^{-1}(x) \cap \Omega} > \theta/2}} \abs{ \pi^{-1}(x) \cap \Omega} dx + \theta/2 \nonumber  \\
& \leq \abs{\set{x \in \Torus^1 : \abs{\pi^{-1}(x) \cap \Omega} > \theta/2}} + \theta/2. \label{ineq:Econtrol}
\end{align}
That is, the set of $x$ such that $E:=\set{x \in \Torus^1 : \abs{\pi^{-1}(x) \cap \Omega} > \theta/2}$ has measure greater than or equal to $\theta/2$. 
Then notice that we may bound the two-dimensional energy in terms of an integral of the one-dimensional energy using Cauchy-Schwarz,
\begin{align*}
E_{ND}[w,u] & \geq \alpha_m h \int_0^1 \int_0^1 \abs{\partial_y w_2 + \frac{1}{2}u_y^2 - \eta}^2 dy dx + h^{3}\int_0^1 \int_0^1 \abs{u_{yy}}^2 dy dx \\ & \quad+ \alpha_s \left(\int\int_{\Omega} \abs{\partial_y w_2}^2 dy dx\right)^{1/2} \left(\int\int_\Omega \abs{w_2}^2 dy dx\right)^{1/2} \\ 
& \geq \alpha_m h \int_0^1 \int_0^1 \abs{\partial_y w_2 + \frac{1}{2}u_y^2 - \eta}^2 dy dx + h^{3}\int_0^1 \int_0^1 \abs{u_{yy}}^2 dy dx
 \\ & \quad + \alpha_s \int_0^1 \left(\int_{\pi^{-1}(x) \cap \Omega} \abs{\partial_y w_2}^2 dy\right)^{1/2} \left(\int_{\pi^{-1}(x) \cap \Omega} \abs{w_2}^2 dy\right)^{1/2} dx.
\end{align*}
It hence follows from \eqref{ineq:Econtrol} and \eqref{ineq:lowerbd1d} of Theorem \ref{thm:oneDLoc} (i) that
\begin{align*} 
E_{ND}[w,u] & \geq \int_E \alpha_mh \int_0^1 \abs{\partial_y w_2 + \frac{1}{2}u_y^2 - \eta}^2 dy + h^3  \int_0^1 \abs{u_{yy}}^2 dy dx 
\\ & \quad + \int_E \alpha_s \left(\int_{\pi^{-1}(x) \cap \Omega} \abs{\partial_y w_2}^2 dy\right)^{1/2} \left(\int_{\pi^{-1}(x) \cap \Omega} \abs{w_2}^2 dy\right)^{1/2} dx \\
& \geq K_1 \frac{\theta}{2}\min\left(\frac{1}{4}\alpha_m h \eta^{2}\theta^2, \alpha_s^{2/3} h\eta^{5/3} \alpha_s^{2/3} \frac{\theta^{5/3}}{2^{5/3}(1-\frac{1}{2}\theta)^{1/3}}\right),
\end{align*}
from which the lower bound \eqref{ineq:lowerbound2d} of Theorem \ref{thm:2d} (i) follows. 
\end{proof} 

\subsection{Upper Bound} \label{sec:2DUpper}
\begin{proof}(Theorem \ref{thm:2d} (ii),(iii))
First we consider the upper bounds in (ii), which are well understood.  
The first upper bound in \eqref{ineq:upperbound2d1} can be realized by the trivial choice of $w = u = 0$, which corresponds to the case of no blistering at all.
The second upper bound in \eqref{ineq:upperbound2d1} follows from the constructions of \cite{JinSternberg01,BelgacemContiSimoneMuller00}, choosing $w = 0$ in the bonded region.

We now turn to the upper bound \eqref{ineq:upperbound2d2} in (iii), for which we use an entirely new construction consisting of a periodic lattice of blisters. 
We divide the construction into three steps; our argument is analogous to the construction of the `origami' patterns in \cite{ContiMaggi08}.
The first step is to build a piecewise linear deformation with sharp folds and exactly zero membrane energy.
Since such a deformation has infinite bending energy, the next step is to use a minimal ridge construction to smooth these folds in a way which achieves the optimal scaling law for folds in F\"oppl-von K\'arm\'an, determined by \cite{Venkataramani03}.  
This step is related to several existing works \cite{ContiMaggi08,Lobkovsky96,AudolyBoudaoudIII08} (see also \cite{AudolyPomeau,Venkataramani03}); our treatment adapts the fully nonlinear construction of \cite{ContiMaggi08}, since it is the cleanest and most flexible approach. 
In the third step, we connect the folds with a procedure similar to that used in \cite{ContiMaggi08}. Then, we compute the total energy of the periodic construction and determine the optimal length-scale $l_2$ which achieves \eqref{ineq:upperbound2d2}.  
\\

\noindent
{\it Step 1: Piecewise linear, zero-membrane energy deformation:}\\
We divide the unit square into periodic cells of side-length $l$, to be chosen later.
In each cell we take the region $[0,\sqrt{\theta} l]\times [0,\sqrt{\theta} l]$ as bonded and the complementary region as debonded, making a blistered region which is in the shape of a sideways `L'. See Figure \ref{fig:MainFold} for a depiction of one cell and Figure \ref{fig:Periodic2D} for several periodic cells laid out. 
For simplicity, for the rest of the proof we will largely ignore the dependence on the area fraction $\theta$. 
In this section we denote the construction $(\bar w, \bar u) = (\bar w_1, \bar w_2, \bar u)$. 

In the bonded region, $\bar u = 0$ and we eliminate the membrane energy by choosing $\bar w = (\bar w_1,\bar w_2) = (\eta x,\eta y)-(\eta x_j,\eta y_j)$, where $(x_j,y_j)$ is the center of the bonded region to make $w$ average zero over the bonded region. 
It follows that the energy in the bonded region is due only to the substrate energy, which is computed to be:
\begin{equation}
\alpha_s \norm{\grad \bar w}_{L^2([0,\sqrt{\theta}l]^2)}\norm{\bar w}_{L^2([0,\sqrt{\theta}l]^2)} \sim \alpha_s\eta^{2} l^3. \label{ineq:substrateEst}
\end{equation}
Next we consider the L-shaped blister on the sides of the periodic cell. 
For now, ignore the corner and consider one of the sides, say the vertical strip, which is of width $(1-\sqrt{\theta})l$. 
Divide the strip into two rectangles vertically, R1 and R2, as shown in Figure \ref{fig:MainFold}. 
\begin{center}
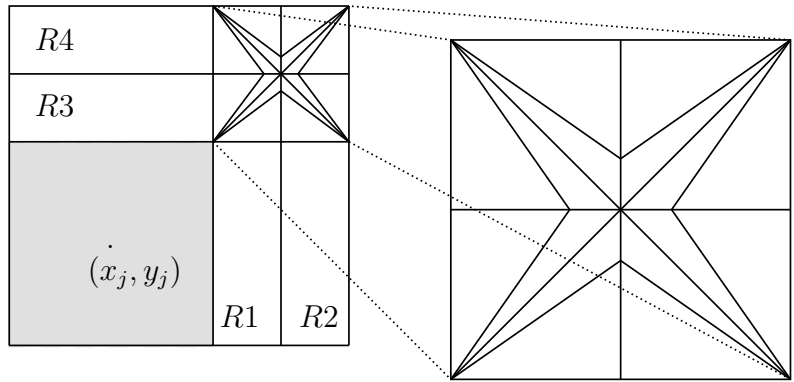
\begin{figure}[hbpt] 
\begin{picture}(150,150)(-100,0)
\put(80,20){$R1$}
\put(110,20){$R2$}
\put(10,100){$R3$}
\put(10,125){$R4$}
\put(37,50){$.$}
\put(30,38){$(x_j,y_j)$}
\scalebox{0.80}{\includegraphics{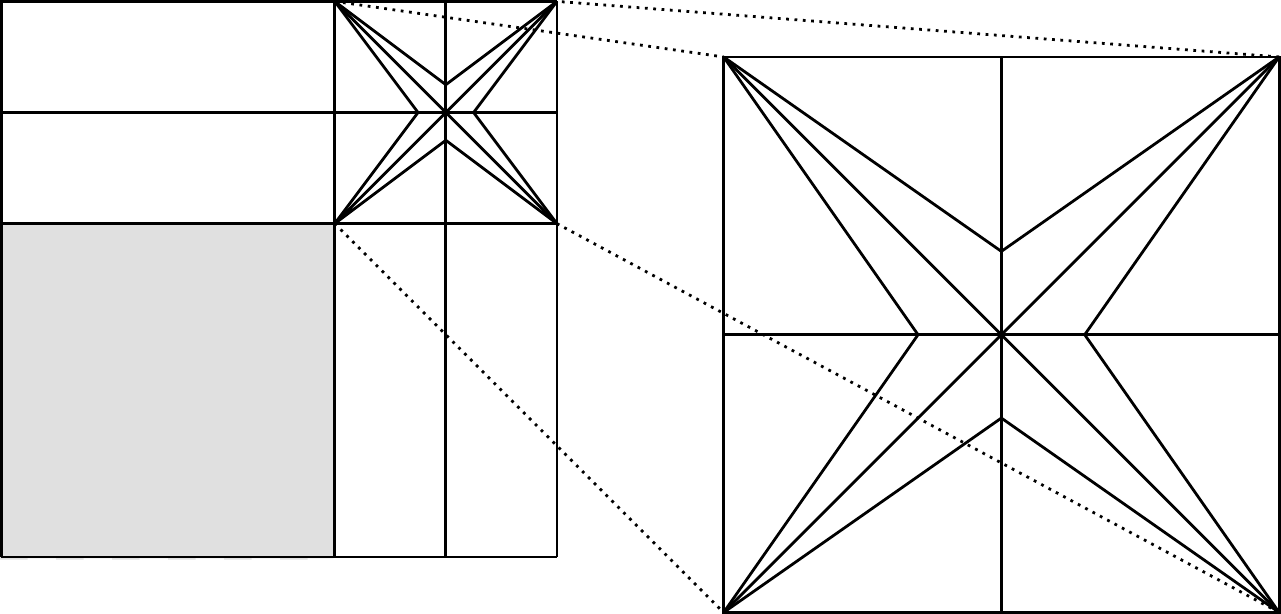}}
\end{picture}
\caption{A periodic cell of side-length $l$ viewed from the top. At this stage in the construction, the black lines represent infinitely thin folds and the displacement $(\bar w,\bar u)$ has exactly zero membrane energy. The shaded region is bonded to the substrate.}  
\label{fig:MainFold}
\end{figure}
\end{center}
\begin{center}
\begin{figure}[hbpt] 
\begin{picture}(150,190)(-100,0)
\scalebox{0.60}{\includegraphics{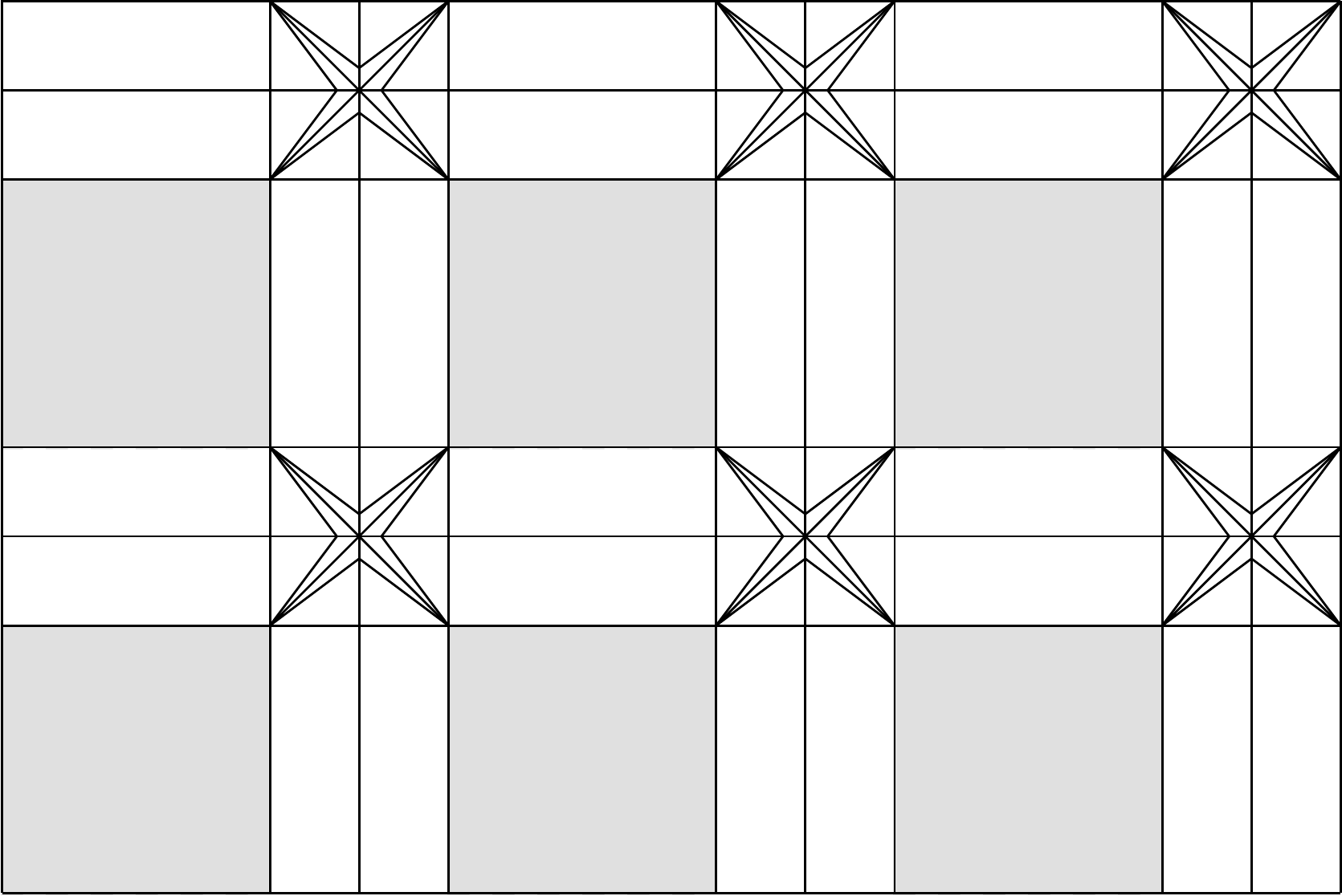}}
\end{picture}
\caption{Six periodic cells used in the construction for the proof of Theorem \ref{thm:2d} (iii).} 
\label{fig:Periodic2D}
\end{figure}
\end{center}
In R1 we set 
\begin{align*} 
\grad \bar w_1 & = \left(-\eta \frac{\sqrt{\theta}}{1-\sqrt{\theta}},0 \right),\\
\grad \bar w_2 & = (0,\eta),\\
\grad \bar u   & = \left(\sqrt{2\eta \left( 1 + \frac{\sqrt{\theta}}{1-\sqrt{\theta}} \right)},0\right), 
\end{align*} 
and in R2 we set
\begin{align*} 
\grad \bar w_1 & = \left(-\eta \frac{\sqrt{\theta}}{1-\sqrt{\theta}},0 \right),\\
\grad \bar w_2 & = (0,\eta),\\
\grad \bar u   & = \left(-\sqrt{2\eta \left( 1 + \frac{\sqrt{\theta}}{1-\sqrt{\theta}} \right)},0\right). 
\end{align*} 
It follows that the membrane energy vanishes in both R1 and R2.
Similarly we can set $(\bar w,\bar u)$ in R3 and R4 by reflecting across $x = y$.  
Strictly speaking, it is not necessary to use folds in the strips; for example, we could use the construction in Section \ref{sec:UppBd1D}.  
However, as the corner construction uses folds, we may as well use them here as well, as the additional folds will only affect the prefactor in the scaling law. 

We now turn our attention to the corner square $[\sqrt{\theta} l,l]\times [\sqrt{\theta} l,l]$.
Here there are relatively few constraints that must be satisfied by the test function: simply that the film connects continuously with the rest of the construction on the edges of the square and satisfies zero membrane energy almost everywhere.  
Since this sub-problem is so unconstrained and has a high degree of redundancy, it is specifically here that \eqref{def:VarP} seems to share some similarities with the crumpling problem. 
Indeed, the construction that follows is close to the constructions in \cite{ContiMaggi08}. 

We determine a suitable $(w,u):[-1,1]\rightarrow \Real^2 \times [0,\infty)$ which we use to define $(\bar w,\bar u)$ on $[\sqrt{\theta}l, l] \times [\sqrt{\theta} l, l]$ via the change of variables: 
\begin{subequations} \label{def:RealSqr}
\begin{align}
(\bar{x},\bar{y}) &= \left( \frac{l}{2} + \frac{l}{2}\sqrt{\theta}, \frac{l}{2} + \frac{l}{2}\sqrt{\theta} \right), \\ 
\bar{w}(x,y) & = \frac{l}{2}(1-\sqrt{\theta})w\left(\frac{2(x - \bar{x})}{(1-\sqrt{\theta})l},\frac{2(y - \bar{y})}{(1-\sqrt{\theta})l}\right) + (\eta (x - \bar{x}), \eta (y - \bar{y})), \label{def:RealSqr2} \\  
\bar{u}(x,y) & = \frac{l}{2}(1-\sqrt{\theta})u\left(\frac{2(x - \bar{x})}{(1-\sqrt{\theta})l},\frac{2(y - \bar{y})}{(1-\sqrt{\theta})l}\right). 
\end{align}
\end{subequations}
As with $(\bar w,\bar u)$, we choose $(w,u)$ to be continuous and piecewise linear. To ensure that $(\bar w,\bar u)$ satisfies zero membrane energy almost everywhere we see from \eqref{def:RealSqr2} that we must have
\begin{align}
\abs{e(w) + \frac{1}{2}\grad u \otimes \grad u} = 0 \textup{ a.e.}. \label{cond:ZeroMemb}
\end{align} 
Moreover, the choice \eqref{def:RealSqr} and the definition of $(\bar w, \bar u)$ in R1-4 imposes boundary conditions on $(w,u)$, which after defining the useful quantity 
\begin{align*} 
\alpha = & = \eta\left( 1 + \frac{\sqrt{\theta}}{1 - \sqrt{\theta}}\right),
\end{align*} 
are written as
\begin{subequations} \label{def:wuBdy}
\begin{align}
w(x,y) & = -(\alpha x,\alpha y), \quad\quad \textup{for } (x,y) \in \partial [-1,1]^2 \\ 
\partial_x u(x,y) & = \sqrt{2\alpha}, \quad\quad\quad\quad \textup{for } y = \pm 1, \;\; x \in (-1,0) \\
\partial_x u(x,y) & = -\sqrt{2\alpha}, \quad\quad\quad \textup{for } y = \pm 1, \;\; x \in (0,1) \label{def:wuBdyuX} \\
\partial_y u(x,y) & = \sqrt{2\alpha}, \quad\quad\quad\quad \textup{for } x = \pm 1, \;\; y \in (-1,0) \\
\partial_y u(x,y) & = -\sqrt{2\alpha}, \quad\quad\quad \textup{for } x = \pm 1, \;\; y \in (0,1) \\
u(-1,-1) & = u(-1,1) = u(1,-1) = u(1,1) = 0. \label{def:wuBdyu0}
\end{align}
\end{subequations}

The boundary conditions \eqref{def:wuBdy} and the condition \eqref{cond:ZeroMemb} are the only constraints we have on $(w,u)$, so there is a certain amount of freedom. Our construction is chosen to have folds at all of the solid black lines shown in the right half of Figure \ref{fig:MainFold}.
We choose the following reflection symmetries:
\begin{subequations} \label{cond:symsfull}
\begin{align}
(w_1(x,y),w_2(x,y),u(x,y)) &= (-w_1(-x,y),w_2(-x,y),u(-x,y)) \\
(w_1(x,y),w_2(x,y),u(x,y)) &= (w_1(x,-y),-w_2(x,-y),u(x,-y)) \\ 
(w_1(x,y),w_2(x,y),u(x,y)) &= (-w_2(-y,-x),-w_1(-y,-x),u(-y,-x)) \\ 
(w_1(x,y),w_2(x,y),u(x,y)) &= (w_2(y,x),w_1(y,x),u(y,x)),  
\end{align} 
\end{subequations}
which correspond to natural reflections across the vertical, horizontal and diagonals respectively. 
By continuity, the symmetries imply
\begin{subequations} \label{cond:syms}
\begin{align}
w_1(0,y) & = 0, \\
w_2(x,0) & = 0, \\
w_1(x,-x) + w_2(x,-x) & = 0, \\ 
-w_1(x,x) + w_2(x,x) & = 0. 
\end{align}
\end{subequations}
Due to our choice of symmetries, it suffices to construct the $(w,u)$ in the right triangle defined by $C = (0,0)$, $A = (0,-1)$ and $B = (1,-1)$. 
The triangle we are considering here can be seen in Figure \ref{fig:CloseUp}. 
\begin{center}
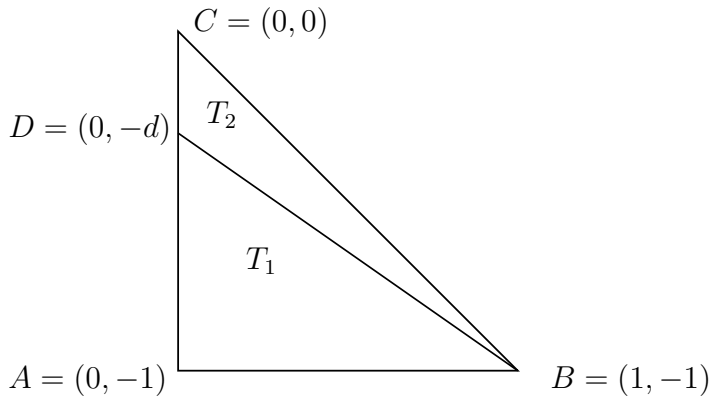
\begin{figure}[hbpt] 
\begin{picture}(150,150)(-100,0)
\put(115,135){$C = (0,0)$}
\put(45,95){$D = (0,-d)$}
\put(135,45){$T_1$}
\put(120,100){$T_2$}
\put(45,0){$A=(0,-1)$}
\put(250,0){$B = (1,-1)$}
\scalebox{0.80}{\includegraphics{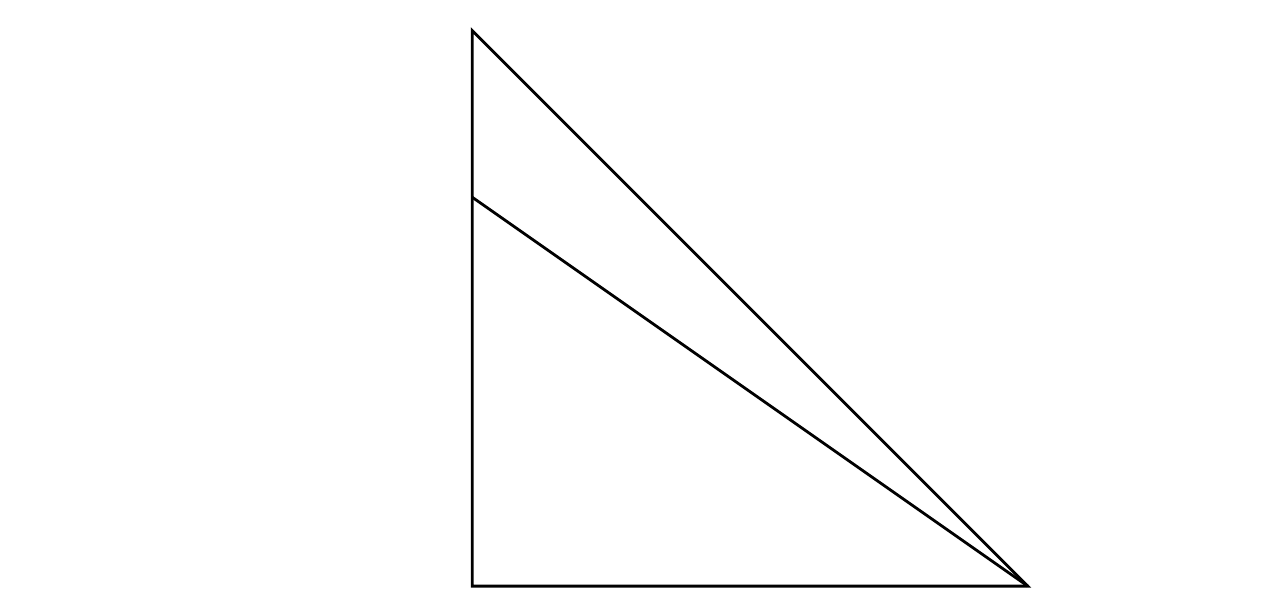}}
\end{picture}
\caption{The renormalized triangle of the corner construction. The lines indicate infinitely thin folds.}
\label{fig:CloseUp}
\end{figure}
\end{center}
Let $d = (\sqrt{2} - 1)^{2} = 3 - 2\sqrt{2}$ and define a fourth point $D = (0,-d)$.
It turns out that this is the unique choice of $d$ for which the following construction is possible. 
Define two triangles, the right triangle defined by $BDA$ denoted $T_1$ and the triangle defined by $DCB$ denoted $T_2$ (see Figure \ref{fig:CloseUp}). 
We use linear interpolation to define $w$ and $u$, which is equivalent to 
the gradients being constant in the interiors of $T_1$ and $T_2$.   
The boundary conditions \eqref{def:wuBdy} and symmetries \eqref{cond:syms} constrain the values of $w$ at $A$, $B$ and $C$, and we make the choice
\begin{align*}
w(D) = (0,\alpha). 
\end{align*} 
By linear interpolation, this defines the in-plane displacement in the interiors of both triangles. 
Now we choose $u$ to eliminate the membrane term in the F\"oppl-von K\'arm\'an formulation with $\eta = 0$: 
\begin{align*} 
\int_{T_1 \cup T_2} \abs{e(w) + \frac{1}{2}\grad u \otimes \grad u}^2dx dy & = \int_{T_1 \cup T_2} \abs{\partial_x w_1 + \frac{1}{2}\abs{\partial_x u}^2}^2 dx dy \\ 
& \quad + \int_{T_1 \cup T_2} \abs{\partial_y w_2 + \frac{1}{2}\abs{\partial_y u}^2}^2 dx dy \\
& \quad + \frac{1}{4}\int_{T_1 \cup T_2} \abs{\partial_x w_2 + \partial_y w_1 + \partial_x u \partial_y u}^2 dx dy. 
\end{align*}
By \eqref{def:wuBdyu0} we have $u(B) = 0$, as the film is bonded at the corner.  
Moreover, by \eqref{def:wuBdyuX}, $u(A) = \sqrt{2\alpha}$. 
In order to eliminate the $x$ contribution to the membrane energy in $T_1$ we choose
\begin{align*}
\grad w_1 = (-\alpha,0), \quad \grad w_2 = (0,0),\quad \grad u = (-\sqrt{2\alpha},0). 
\end{align*} 
This in turn implies $u(D) = \sqrt{2\alpha}$ and it follows quickly that the membrane energy vanishes on $T_1$. 
In $T_2$ we know from the definition of $w$, denoting $e_{BC} := 2^{-1/2}(-1,1)$,  
\begin{align*} 
\grad w_1 \cdot e_y & = 0, \\ 
\grad w_2 \cdot e_y & = \frac{-\alpha}{d}, \\ 
\grad w_1 \cdot e_{BC} &= \frac{\alpha}{\sqrt{2}}, \\ 
\grad w_2 \cdot e_{BC} &= -\frac{\alpha}{\sqrt{2}},
\end{align*} 
which implies
\begin{align*}
\partial_x w_1 &= -\alpha, \\
\partial_x w_2 &= \alpha - \frac{\alpha}{d}. 
\end{align*}
Eliminating the $x$ and $y$ contributions to the membrane energy dictates the choice 
\begin{align*} 
\grad u = \left(\sqrt{2\alpha},\sqrt{\frac{2\alpha}{d}}\right), 
\end{align*}
which implies
\begin{align*}
u(C) = u(D) + d\grad u \cdot e_y =  \sqrt{2\alpha} + \sqrt{2\alpha d} = \sqrt{2 \alpha}\left(\frac{1}{\sqrt{d}} -1\right) = \sqrt{2}\grad u \cdot e_{BC}, 
\end{align*}
where the third equality follows from the specific choice of $d$. 
It remains to verify that the entire membrane energy vanishes on $T_2$: 
\begin{align*} 
\partial_x w_1 + \frac{1}{2}u_x^2  = -\alpha + \frac{1}{2}2 \alpha & = 0. \\ 
\partial_y w_2 + \frac{1}{2}u_y^2  = -\frac{\alpha}{d} + \frac{1}{2}\frac{2\alpha}{d} & = 0. \\ 
\partial_x w_2 + \partial_y w_1 + u_x u_y  = \alpha - \frac{\alpha}{d} + \sqrt{2\alpha} \sqrt{\frac{2\alpha}{d}} & = \alpha\left( \frac{d - 1 + 2\sqrt{d}}{d}\right) \\ 
& = \frac{ (\sqrt{2} - 1)^2 - 1 + 2\sqrt{2} - 2}{d} \\
& = 0.   
\end{align*} 
By the symmetries \eqref{cond:symsfull} we then define $(u,w)$ on the remaining 7 triangles. 
At this point we can then use \eqref{def:RealSqr} to define $(\bar w, \bar u)$ in the square $[\sqrt{\theta}l, l] \times [\sqrt{\theta} l, l]$, which completes the construction of a piecewise linear, zero-membrane energy deformation with infinitely thin folds. 
\\

\noindent
{\it Step 2: Minimal Ridge Smoothing}\\
Next, we must refine the preceding construction into one which has finite bending energy. 
The primary step in this direction is to smooth the infinitely thin folds with a minimal ridge which obtains the optimal scaling law. We also must smooth at the vertices, but this is straightforward and is carried out in Step 3.

The minimal ridge results from a balance of bending and membrane energy along the fold. 
An `obvious' smoothing approach of rounding out the fold over a fixed thickness is sub-optimal, and instead a smoothed region which is variable along the length of the fold turns out to do better. 
The following lemma is an adaptation of Lemma 2.1 in \cite{ContiMaggi08} (the nonlinear construction there must be modified to fit into the F\"oppl-von K\'arm\'an framework).
We must know precisely how large a smoothing region is necessary to obtain the optimal scaling law determined in \cite{Venkataramani03}  (optimal with respect to both $\eta$ and $h$). 
Condition \eqref{cond:sigma} in Lemma \ref{lem:MinRidge} below will determine \eqref{cond:2d1b} by placing a lower bound on the size of the construction which ensures there is room to choose the smoothing parameter as \eqref{def:sigma}. 
Condition \eqref{cond:inplane} determines \eqref{cond:2d2} which ensures that the nonlinear construction of \cite{ContiMaggi08} can still achieve the optimal scaling if properly adapted to the F\"oppl-von K\'arm\'an energy.
As in \cite{ContiMaggi08}, we consider the deformation of the quadrilateral $[abcd]$ given in Figure \ref{fig:MinRidge}.  
Unlike the fully nonlinear model, the F\"oppl-von K\'arm\'an energy is not invariant with respect to rotations so we must state the lemma for a general piecewise linear transformation. 
\begin{center}
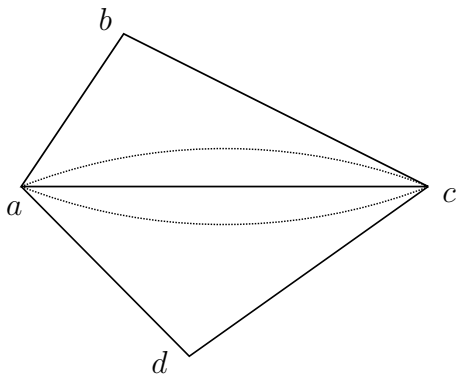
\begin{figure}[hbpt] 
\begin{picture}(150,150)(-100,0)
\put(65,60){$a$}
\put(230,65){$c$}
\put(120,0){$d$}
\put(100,130){$b$}
\scalebox{0.80}{\includegraphics{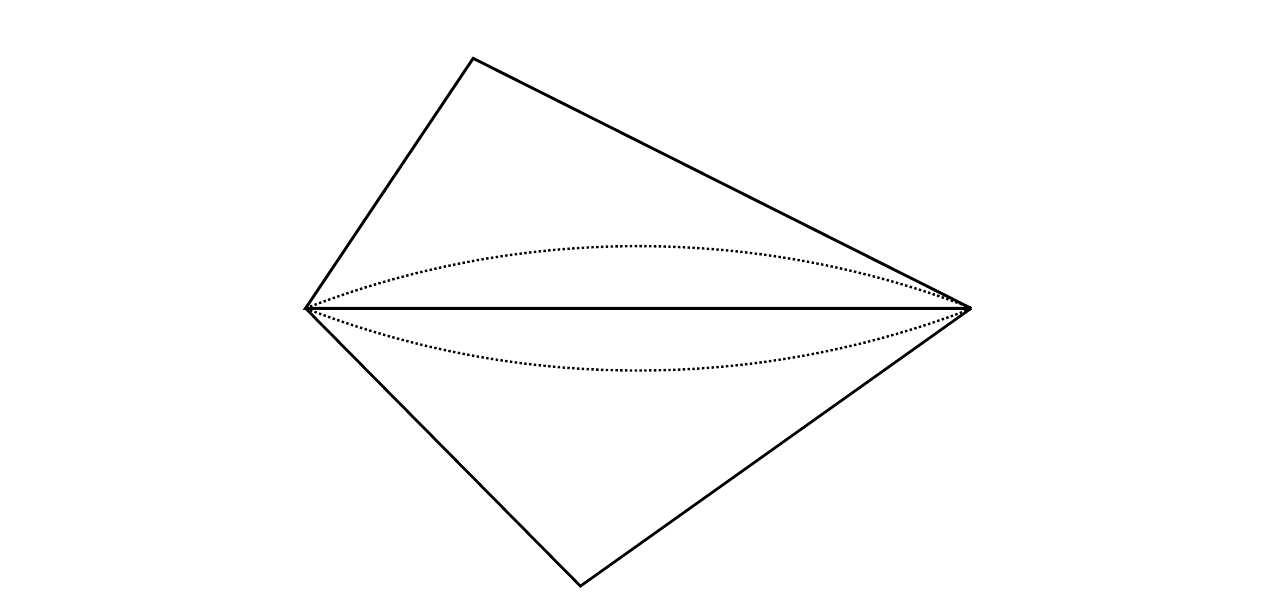}}
\end{picture}
\caption{The quadrilateral considered in Lemma \ref{lem:MinRidge}. The dashed lines denote $y = \pm f(x)$, which encloses the region where smoothing is added (except near the vertices where a different smoothing takes place - see Step 3).}  
\label{fig:MinRidge}
\end{figure}
\end{center}
\begin{lemma}[F\"oppl-von K\'arm\'an Minimal Ridge] \label{lem:MinRidge}
Consider the quadrilateral $[abcd]$ (see Figure \ref{fig:MinRidge}) such that the length of $\abs{ac} = l$ and $\abs{bd} \sim l$ (hence the triangles are not badly skewed, which is equivalent to the parameter $\tau$ in \cite{ContiMaggi08} being $O(1)$).  
Suppose there exists a continuous, piecewise linear $(\hat{w}_1, \hat{w}_2,\hat{u})$ which satisfies 
\begin{align*}
\abs{\partial_x \hat{w}_1 + \frac{1}{2}(\partial_x \hat{u})^2 - \eta} + \abs{\partial_y \hat{w}_2 + \frac{1}{2}(\partial_y \hat{u})^2- \eta} + \abs{\partial_x \hat{w}_2 + \partial_y \hat{w}_1 + \partial_x \hat{u} \partial_y \hat{u}} = 0 \textup{ in } [acd] \textup{ and } [abc].  
\end{align*}
Define $\alpha_L =\partial_y \hat{u}|_{[acd]}$, $\alpha_R = \partial_y \hat{u}|_{[abc]}$
and denote $\phi = \max\left(\abs{\alpha_L}, \abs{\alpha_R}\right)$ which we assume is strictly positive and $\phi \leq 1$ ($\phi$ is a bound for the angle of the fold).
We fix a smoothing parameter $\sigma < l/8$  
and assume  
\begin{equation}
\phi^4 \lesssim \frac{\sigma^{2/3}}{l^{2/3}}. \label{cond:inplane}
\end{equation} 
Define the region $[abcd]_\sigma = [abcd] \setminus \left(B(a,\sigma) \cup B(c,\sigma)\right)$. 
Then there exists a deformation $(w_1,w_2,u)$ which satisfies 
\begin{align*} 
\norm{w_i - \hat{w}_i}_{L^\infty([abcd])} + \norm{u - \hat{u}}_{L^\infty([abcd])} & \lesssim \phi l^{2/3}\sigma^{1/3} \\ 
\norm{\grad u - \grad\hat{u}}_{L^\infty([abcd])} & \lesssim  \phi \\ 
\norm{\grad w_i - \grad \hat{w}_i}_{L^\infty([abcd])} & \lesssim \phi^2 \\
\norm{D^2u}_{L^\infty([abcd]_\sigma)} & \lesssim \frac{\phi}{\sigma},
\end{align*}
and 
\begin{align*}
\alpha_m h \int_{[abcd]_\sigma} \abs{e(w) - \frac{1}{2}\grad u \otimes \grad u - \eta I}^2 dx + h^3 \int_{[abcd]_\sigma} \abs{D^2u}^2 dx \lesssim h\left[\alpha_m \phi^4 \sigma^{5/3}l^{1/3} + h^2\frac{\phi^2 l^{1/3}}{\sigma^{1/3}} \right]. 
\end{align*}   
Moreover, 
\begin{align*} 
\left(w,u,\grad w, \grad u\right) = \left(\hat w,\hat u,\grad \hat w, \grad \hat u\right) \textup{ on } \partial[abcd]. 
\end{align*}
Choosing the optimal $\sigma$, specifically, 
\begin{align}
\sigma = \frac{h}{\alpha_m^{1/2} \phi}, \label{def:sigma}
\end{align}
we get the estimate
\begin{align}
\alpha_m h \int_{[abcd]_\sigma} \abs{e(w) - \frac{1}{2}\grad u \otimes \grad u - \eta I}^2 dx + h^3 \int_{[abcd]_\sigma} \abs{D^2u}^2 dx \lesssim \alpha_m^{1/6} \phi^{7/3} l^{1/3} h^{8/3}.  \label{ineq:optEnergy}
\end{align}
For \eqref{def:sigma} and \eqref{ineq:optEnergy} to be valid we require
\begin{align}
\frac{h}{\alpha_m^{1/2} \phi} < c l, \label{cond:sigma}
\end{align}
for a sufficiently small constant $c > 0$. 
\end{lemma}
\begin{proof} 
The set up is similar to Lemma 2.1 in \cite{ContiMaggi08} with several changes to accommodate the lack of rotational invariance and the form of the F\"oppl-von K\'arm\'an membrane energy. 

First, translate the coordinate system and assume without loss of generality that the vertex $a$ is located at $(0,0)$
and that $\hat w_1(a) = \hat w_2(a) = \hat{u}(a) = 0$. 
We may replace the membrane energy with 
\begin{align*} 
\alpha_m h \int_{[abcd]_\sigma} \abs{e(w) - \frac{1}{2}\grad u \otimes \grad u}^2 dx, 
\end{align*}
by replacing $(\hat w_1,\hat w_2)$ by $(\hat w_1,\hat w_2) - \eta(x,y)$, which does not change $\hat u$ and does not change the meaning of $\phi$ or the content of Lemma \ref{lem:MinRidge}. 
Similarly, for the remainder of the proof we assume 
\begin{align*}
\abs{\partial_x \hat{w}_1 + \frac{1}{2}(\partial_x \hat{u})^2} + \abs{\partial_y \hat{w}_2 + \frac{1}{2}(\partial_y \hat{u})^2} + \abs{\partial_x \hat{w}_2 + \partial_y \hat{w}_1 + \partial_x \hat{u} \partial_y \hat{u}} = 0 \textup{ in } [acd] \textup{ and } [abc].
\end{align*}
Next, notice by continuity that the $x$ derivatives of $\hat w_1, \hat w_2, \hat u$ must be the same in both triangles. That is, 
\begin{align*} 
\partial_x \hat u & := \partial_x \hat u|_{[acd]} = \partial_x \hat u|_{[abc]} \\
\partial_x \hat w_1 &:= \partial_x \hat w_1|_{[acd]} = \partial_x \hat w_1|_{[abc]} \\ 
\partial_x \hat w_2 & := \partial_x \hat w_2|_{[acd]} = \partial_x \hat w_2|_{[abc]}.
\end{align*} 
Replacing $(\hat w_1,\hat w_2)$ by $(\hat w_1 + \alpha y, \hat w_2 - \alpha x)$, we may assume without loss of generality that $\partial_x \hat w_2 = 0$.  

As in the nonlinear case, $\sigma$ is a smoothing parameter which corresponds to the characteristic width of the fold. 
The work of \cite{Lobkovsky96,Venkataramani03} predicts that the width of the fold should vary along the length, and this is also the case in \cite{ContiMaggi08} and here. 
As in \cite{ContiMaggi08} define 
\begin{align} 
f_0(x) = \tau \sigma^{1/3}(x+\sigma)^{2/3} - \tau \sigma, \label{def:f0}
\end{align}
where $\tau$ is the minimum of the absolute values of slopes of the lines $[ab],[ad],[dc]$ and $[bc]$.
We are assuming $\tau = O(1)$ and we will henceforth not further concern ourselves with it.
 As in \cite{ContiMaggi08} we may define a smooth $f$ which obeys similar estimates as $f_0$, in particular (2.30) in \cite{ContiMaggi08} which implies $\partial_x^nf_0 \sim \partial_x^nf$ for $0 \leq n \leq 3$; we omit the details.
This $f$ will define the region in which smoothing is taking place: the width of the fold as a function of $x$. Figure \ref{fig:MinRidge} shows the smoothing region denoted by the dotted lines roughly as defined by $f(x)$.
Define the region $D := \set{(x,y)\in [abcd]: -f(x) < y < f(x)}$. 
This is the only region where the construction alters $(\hat{w}_1,\hat{w}_2,\hat{u})$.

The construction in \cite{ContiMaggi08} has two additional components: the curve $\gamma(t):[-1,1] \rightarrow \Real^3$ which is used to define the deformation in the $y$ direction and $\beta(x,y):\Real^2 \rightarrow \Real$ which corrects the deformation in the $x$ direction.
To define the curve $\gamma$, Lemma 2.2 in \cite{ContiMaggi08} must be adapted to the F\"oppl-von K\'arm\'an setting, carried out below in Lemma \ref{lem:gammachoice} and applied with $\alpha_L = \partial_y \hat{u}|_{[acd]}$ and $\alpha_R = \partial_y \hat{u}|_{[abc]}$ to produce a curve $\gamma$ which is consistent with the deformation in the region $(x,y) \notin D$. 
The only non-trivial difference between Lemma 2.2 in \cite{ContiMaggi08} and Lemma \ref{lem:gammachoice} is that the condition
\begin{align} 
\gamma_2^\prime(t) + \frac{1}{2}(\gamma^\prime_3(t))^2 - 1 & = 0 \label{cond:FvKGamma}  
\end{align}
replaces the arclength parametrization condition in \cite{ContiMaggi08} (the former is the F\"oppl-von K\'arm\'an analogue of the latter).

In \cite{ContiMaggi08}, $\beta(x,y)$ is defined in terms of $\gamma$ via equation (2.18) of \cite{ContiMaggi08}. 
Unfortunately, since \eqref{cond:FvKGamma} is not equivalent to being unit-speed, if we define $\beta$ in the same fashion here we will not have $\beta(x,y) = x$ at $\partial D$ (we cannot guarantee that $\zeta$ in \cite{ContiMaggi08} integrates to zero), which is necessary for the minimal ridge to be consistent with the $(\hat w,\hat u)$ outside $D$.    
This can be corrected by removing the defect; we therefore define
\begin{align}
\beta(x,y) = x - f(x)f^\prime(x) \int_{-1}^{\frac{y}{f(x)}} \gamma^\prime(t) \cdot \eta(t) dt + \frac{y + f(x)}{2}f^\prime(x) \int_{-1}^{1} \gamma^\prime(t) \cdot \eta(t) dt,  \label{def:beta}
\end{align}
where as in \cite{ContiMaggi08}, $\eta(t) = \gamma - t\gamma^\prime(t)$. 
Defining $E = \int_{-1}^{1} \gamma^\prime(t) \cdot \eta(t) dt$, we have
\begin{align*}
E = \int_{-1}^{1} \gamma^\prime(t) \cdot \eta(t) dt & = \int_{-1}^{1} \gamma^\prime(t) \cdot \left(\gamma - t\gamma^\prime(t)\right) dt \\ 
  & = \int_{-1}^{1} \gamma_2^\prime(t)\gamma_2(t) + \gamma_3^\prime(t)\gamma_3(t) - t\abs{\gamma_2^\prime(t)}^2 - t\abs{\gamma_3^\prime(t)}^2 dt \\ 
  & = \frac{1}{2}\abs{\gamma_2(1)}^2 - \frac{1}{2}\abs{\gamma_2(-1)}^2 + \frac{1}{2}\abs{\gamma_3(1)}^2 - \frac{1}{2}\abs{\gamma_3(-1)}^2  \\ 
& \quad + \int_{-1}^1 t\left(\abs{\gamma_2^\prime(t)}^2 - 2 + 2\gamma_2^\prime + \abs{\gamma_3^\prime(t)}^2 + 2 - 2\gamma_2^\prime \right) dt \\ 
  & = \frac{1}{2}\abs{\gamma_2(1)}^2 - \frac{1}{2}\abs{\gamma_2(-1)}^2 + \frac{1}{2}\abs{\gamma_3(1)}^2 - \frac{1}{2}\abs{\gamma_3(-1)}^2   \\ 
& \quad + \int_{-1}^1 t\abs{\gamma_2^\prime(t) -1}^2 + t dt, 
\end{align*} 
where the last line follows from \eqref{cond:FvKGamma}. By \eqref{ineq:gammaerror} in Lemma \ref{lem:gammachoice} below, 
\begin{align} 
\abs{E} = \abs{\int_{-1}^{1} \gamma^\prime(t) \cdot \eta(t) dt} \lesssim \phi^4. \label{ineq:betaintegral}
\end{align}
For future reference define
\begin{align*}
\omega(s) = \int_{-1}^{s} \gamma^\prime(t) \cdot \eta(t) dt. 
\end{align*}

With $\gamma$ defined by Lemma \ref{lem:gammachoice} and $\beta$ defined by \eqref{def:beta} we may now define the deformation 
\begin{subequations} \label{def:FvkCM}
\begin{align}
w_1(x,y) & := \left\{ 
\begin{array}{ll} 
\beta(x,y) - x - \partial_x \hat u f(x)\gamma_3\left(\frac{y}{f(x)}\right) + (\partial_x \hat w_1)x & (x,y) \in D \\ 
\hat{w}_1(x,y) & (x,y) \notin D, 
\end{array}
\right. \\
w_2(x,y) & := \left\{
\begin{array}{ll} 
f(x)\gamma_2\left(\frac{y}{f(x)}\right) - y & (x,y) \in D \\ 
\hat{w}_2(x,y) & (x,y) \notin D,  
\end{array}
\right. \\ 
u(x,y) & := \left\{
\begin{array}{ll} 
f(x)\gamma_3\left(\frac{y}{f(x)}\right) + (\partial_x \hat u)x & (x,y) \in D \\ 
\hat{u}(x,y) & (x,y) \notin D. 
\end{array}
\right.
\end{align} 
\end{subequations}
This is similar to Lemma 2.1/2.5 in \cite{ContiMaggi08} except for subtracting $(x,y)$ from the in-plane deformation, the natural adjustment for the F\"oppl-von K\'arm\'an framework, and the adjustments that account for non-zero $x$-derivatives, a complication avoided in \cite{ContiMaggi08} by using the rotation invariance of finite elasticity.  

Assuming Lemma \ref{lem:gammachoice}, we now estimate the energy of \eqref{def:FvkCM}, arguing as in \cite{ContiMaggi08}. 
Computing the $x$-term in the F\"oppl-von K\'arm\'an energy, 
\begin{align*}
\int_{D} \abs{\partial_x w_1 + \frac{1}{2}(\partial_x u)^2}^2 dx dy &  \\
& \hspace{-3cm} = \int_{D} \abs{\beta_x - 1 + \partial_x \hat w_1 - \partial_x \hat u \partial_x\left[f(x)\gamma_3\left(\frac{y}{f(x)}\right)\right] + \frac{1}{2}\left(\partial_x\left[f(x)\gamma_3\left(\frac{y}{f(x)}\right)\right] + \partial_x \hat u \right)^2  }^2 dx dy \\ 
& \hspace{-3cm} = \int_{D} \abs{\beta_x - 1 + \partial_x \hat w_1 + \frac{1}{2}\left(\partial_x\left[f(x)\gamma_3\left(\frac{y}{f(x)}\right)\right]\right)^2 + \frac{1}{2}\left(\partial_x \hat u\right)^2  }^2 dx dy \\ 
& \hspace{-3cm} = \int_{D} \abs{\beta_x - 1 + \frac{1}{2}\left(\partial_x\left[f(x)\gamma_3\left(\frac{y}{f(x)}\right)\right]\right)^2}^2 dx dy. 
\end{align*} 
From \eqref{def:beta} we have, 
\begin{align*}
\int_{D} \abs{\partial_x w_1 + \frac{1}{2}(\partial_x u)^2}^2 dx dy & \\ 
& \hspace{-3cm} \lesssim \int_0^l \int_{-f(x)}^{f(x)} \abs{ - \left( f f^{\prime\prime} + (f^\prime)^2 \right) \omega\left(\frac{y}{f(x)}\right) + \frac{(f^{\prime})^2 + f^{\prime\prime}(y+f)}{2}\omega(1)}^2 dy dx \\ 
 & \hspace{-2.5cm} + \int_0^l \int_{-f(x)}^{f(x)} \abs{ \frac{(f^\prime)^2 y}{f}\left[\gamma_2^\prime\left(\frac{y}{f(x)}\right)\eta_2\left(\frac{y}{f(x)}\right) - \frac{y}{f}(\gamma_3^\prime)^2\left(\frac{y}{f(x)}\right)\right]}^2 dy dx  \\ 
& \hspace{-2.5cm} + \int_0^l \int_{-f(x)}^{f(x)} \abs{\frac{1}{2}(f^\prime)^2 (\gamma_3)^2\left(\frac{y}{f(x)}\right) + \frac{(f^\prime)^2}{2}\left(\frac{y}{f(x)}\right)^2 (\gamma_3^\prime)^2\left(\frac{y}{f(x)}\right) }^2 dy dx. 
\end{align*} 
Using \eqref{ineq:gammaerror} and \eqref{ineq:etaomega} we have, 
\begin{align*} 
\int_{D} \abs{\partial_x w_1 + \frac{1}{2}\partial_x u^2}^2 dx dy & \lesssim \phi^4\int_0^l \int_{-f(x)}^{f(x)}\abs{\abs{f f^{\prime\prime}} + (f^\prime)^2 + \abs{f^{\prime\prime}}\abs{y + f(x)}}^2 dy dx \\ 
& \quad + \phi^4\int_0^l \int_{-f(x)}^{f(x)} \abs{(f^\prime)^2 + \frac{(f^\prime)^2 \abs{y}}{f} + \frac{(f^\prime)^2 \abs{y}^2}{f^2}}^2 dy dx. 
\end{align*} 
Since $f$ and its first few derivatives are comparable to the first few derivatives of \eqref{def:f0}, we may estimate these integrals as in \cite{ContiMaggi08} and deduce
\begin{align*} 
\int_D \abs{\partial_x w_1 + \frac{1}{2}(\partial_x u)^2}^2 dx dy & \lesssim \phi^4 \sigma^{5/3}l^{1/3}.
\end{align*} 
Now we turn to the $y$-term in the membrane energy. 
By \eqref{cond:FvKGamma}, 
\begin{align*}
\int_{D} \abs{\partial_y w_2 + \frac{1}{2}(\partial_y u)^2}^2 dx dy & =  \int_{\Omega} \abs{\gamma^{\prime}_2\left(\frac{y}{f(x)}\right) - 1 + \frac{1}{2}\left(\gamma_3^\prime\left(\frac{y}{f(x)}\right)\right)^2}^2 dx dy = 0. 
\end{align*}
Computing now the last term in the membrane energy, with $\beta(x,y)$ defined \eqref{def:beta}, 
\begin{align*} 
\int \abs{\partial_y w_1 + \partial_x w_2 + \partial_x u \partial_y u}^2 dx dy \\ 
& \hspace{-5cm} = \int \abs{\beta_y(x,y) - \partial_x \hat u \gamma_3^\prime\left(\frac{y}{f(x)}\right) + f^\prime(x) \eta_2\left(\frac{y}{f(x)} \right) + \left(f^\prime(x) \eta_3\left(\frac{y}{f(x)}\right) + \partial_x\hat u  \right)\gamma_3^\prime\left(\frac{y}{f(x)}\right) }^2 dx dy \\ 
& \hspace{-5cm} = \int \abs{\beta_y(x,y) + f^\prime(x) \eta_2\left(\frac{y}{f(x)} \right) + f^\prime(x) \eta_3\left(\frac{y}{f(x)}\right) \gamma_3^\prime\left(\frac{y}{f(x)}\right)}^2 dx dy \\ 
& \hspace{-5cm} = \int \abs{-f^\prime(x)\gamma^\prime\cdot \eta\left(\frac{y}{f(x)}\right) + \frac{1}{2}f^\prime(x)E  + f^\prime(x) \eta_2\left( \frac{y}{f(x)}\right) + f^\prime(x) \eta_3\left( \frac{y}{f(x)}\right) \gamma_3^\prime\left(\frac{y}{f(x)}\right)}^2 dx dy. 
\end{align*}
Then, \eqref{cond:FvKGamma} implies
\begin{align*} 
\int \abs{ \partial_y w_1 + \partial_x w_2 + \partial_x u \partial_y u}^2 dx dy & = \int\abs{\frac{1}{2}f^\prime(x) E + f^\prime(x) \eta_2\left(\frac{y}{f(x)}\right)\left(1 - \gamma_2^\prime\left(\frac{y}{f(x)} \right)\right) }^2 dx dy  \\ 
& \lesssim \int \abs{f^\prime(x) \eta_2\left(\frac{y}{f(x)}\right) \left(\gamma_3^\prime\left(\frac{y}{f(x)} \right) \right)^2}^2 dx dy + \abs{E}^2\int\abs{f^\prime(x)}^2 dx dy.  
\end{align*}  
By Lemma \ref{lem:gammachoice}, \eqref{ineq:betaintegral} and \eqref{def:f0} we may compute (using also that $\phi \leq 1$), 
\begin{align*} 
\int \abs{ \partial_y w_1 + \partial_x w_2 + \partial_x u \partial_y u}^2 dx dy & \lesssim \phi^8 \sigma l.
\end{align*}
Condition \eqref{cond:inplane} implies that $\phi^8 \sigma l \lesssim \phi^4\sigma^{5/3}l^{1/3}$. 
The bending energy computation proceeds similar to Lemma 2.1 \cite{ContiMaggi08}, although simpler 
since $\beta$ no longer plays a role (we only have to consider the term involving $F_1$ in the proof of \cite{ContiMaggi08}).    
Computing the derivatives, denoting $t = y/f(x)$, 
\begin{align*} 
\partial_{yy}u & = \frac{1}{f(x)}\gamma_3^{\prime\prime}(t) \\ 
\partial_{yx}u & = -\frac{f^\prime(x)t}{f(x)}\gamma_3^{\prime\prime}(t) \\
\partial_{xx}u & = f^{\prime\prime}(x) \gamma_3(t) - t f^{\prime\prime}(x) \gamma^\prime_3(t) + t^2\frac{(f^\prime(x))^2}{f(x)}\gamma_3^{\prime\prime}(t).  
\end{align*} 
Defining $D_\sigma = [abcd]_\sigma \cap D$, it follows from Lemma \ref{lem:gammachoice} that 
\begin{align*} 
h^3\int_{D_\sigma} \abs{D^2 u }^2 dx dy & \lesssim h^3 \phi^2 \int_\sigma^{l-\sigma} \int_{-f(x)}^{f(x)} \frac{1}{f(x)^2} +\frac{(f^\prime(x))^2}{f(x)^2} + (f^{\prime\prime}(x))^2 + \frac{(f^\prime(x))^4}{(f(x))^2} dx dy \\ 
& \lesssim h^3 \phi^2 \int_\sigma^{l-\sigma} \int_{-f(x)}^{f(x)} \frac{1}{f(x)^2} + (f^{\prime\prime}(x))^2 + \frac{(f^\prime(x))^4}{(f(x))^2} dx dy. 
\end{align*} 
Following the computations found on page 24 of \cite{ContiMaggi08} we have
\begin{align*} 
h^3 \int_{D_\sigma} \abs{D^2 u }^2 dx dy & \lesssim h^3 \phi^2\frac{l^{1/3}}{\sigma^{1/3}} + h^3\phi^2 \\ 
& \lesssim h^3 \phi^2 \frac{l^{1/3}}{\sigma^{1/3}}, 
\end{align*}
where the last line followed from the requirement that $\sigma < l$. 
Adding the contributions of the bending and membrane completes the lemma. 
\end{proof}

The following is the F\"oppl-von K\'arm\'an equivalent of Lemma 2.2 in \cite{ContiMaggi08}, used in Lemma \ref{lem:MinRidge} above. 

\begin{lemma} \label{lem:gammachoice}
Given $\alpha_L,\alpha_R \in [-1,1]$ there exists $\gamma = (0,\gamma_2(t),\gamma_3(t)) \in C^2([-1,1];\Real^3)$ such that
$\forall\, t \in [-1,1]$,
\begin{align*} 
\gamma_2^\prime(t) + \frac{1}{2}\left(\gamma_3^\prime\right)^2 - 1 = 0, 
\end{align*} 
for $t \in [-1,-2/3)$, 
\begin{align*} 
\gamma(t) = \left(0,\left(1 - \frac{1}{2}\alpha_L^2\right)t,\alpha_L t\right),
\end{align*} 
and for $t \in (2/3,1]$,
\begin{align*}
\gamma(t)  = \left(0,\left(1 - \frac{1}{2}\alpha_R^2\right)t,\alpha_R t\right). 
\end{align*} 
Moreover we may choose $\gamma$ such that $\forall\, t \in [-1,1]$,
\begin{subequations} \label{ineq:gammaerror} 
\begin{align} 
\abs{\gamma_2(t) -\left(1 - \frac{1}{2}\alpha_L^2\right)t} + \abs{\gamma_2(t) -\left(1 - \frac{1}{2}\alpha_R^2\right)t} & \lesssim \max\left(\alpha^2_L,\alpha^2_R\right) \label{ineq:gamma2}  \\ 
\abs{\gamma_3(t) - \alpha_Lt} + \abs{\gamma_3(t) -\alpha_R t} & \lesssim \max\left(\abs{\alpha_L},\abs{\alpha_R}\right) \label{ineq:gamma3} \\ 
\abs{\gamma_2^\prime(t) - \left(1 - \frac{1}{2}\alpha_L^2\right)} + \abs{\gamma_2^\prime(t) - \left(1 - \frac{1}{2}\alpha_R^2\right)} & \lesssim \max\left(\alpha^2_L,\alpha^2_R\right) \label{ineq:gamma2prime} \\
\abs{\gamma_3^\prime(t) - \alpha_L} + \abs{\gamma_3^\prime(t) - \alpha_R} & \lesssim \max\left(\abs{\alpha_L},\abs{\alpha_R}\right) \label{ineq:gamma3prime} \\
\abs{\gamma^{\prime\prime}(t)} & \lesssim \max\left(\abs{\alpha_L},\abs{\alpha_R}\right). \label{ineq:gammaprimeprime}
\end{align}
\end{subequations}
Additionally, if we write $\eta(t) = \gamma(t) - t\gamma^\prime(t)$ and $\omega(t) = \int_{-1}^t \gamma^\prime(s) \cdot \eta(s) ds$ we have $\forall\, t \in [-1,1]$,
\begin{subequations} \label{ineq:etaomega}
\begin{align} 
\abs{\eta_2(t)} = \abs{\gamma_2(t) - t\gamma_2^\prime(t)} & \lesssim \max\left(\alpha^2_L,\alpha_R^2\right) \label{ineq:etaomega1} \\
\abs{\eta_3(t)} = \abs{\gamma_3(t) - t\gamma_3^\prime(t)} & \lesssim \max\left(\abs{\alpha_L},\abs{\alpha_R}\right) \label{ineq:etaomega2} \\ 
\abs{\omega(t)} = \abs{\int_{-1}^t \gamma^\prime(s) \cdot \eta(s) ds} & \lesssim \max\left(\alpha^2_L,\alpha_R^2\right). \label{ineq:etaomega3}
\end{align}
\end{subequations}
\end{lemma} 
\begin{proof} 
The argument is analogous to that of \cite{ContiMaggi08}. 
First define $\bar \gamma(t)$ with 
\begin{align*} 
\bar \gamma(t) & = \left\{ 
\begin{array}{ll} 
\left(1 - \frac{1}{2}\alpha_L^2\right)t e_2 + \alpha_L t e_3 & t \in [-1,0] \\ 
\left(1 - \frac{1}{2}\alpha_R^2\right)t e_2 + \alpha_R t e_3 & t \in (0,1].  
\end{array}
\right.
\end{align*}
Let $\rho(\xi)$ be a smooth, non-negative, even function such that $\int \rho = 1$ and $\rho(\xi) = 0$ for $\abs{\xi} \geq \abs{\frac{1}{3}}$. 
Define now $\tilde\gamma = \rho \ast \bar\gamma$.
The new $\tilde\gamma$ is smooth and satisfies \eqref{ineq:gammaerror} but not \eqref{cond:FvKGamma}, which is the analogue of the unit speed condition in \cite{ContiMaggi08}.  
However, if we have $\gamma_3$, we could construct the suitable $\gamma_2$ via integration provided we have the consistency condition
\begin{align} 
\frac{1}{2}\int_{-1}^1 \left(\gamma_3^{\prime}(t)\right)^2 dt = 2 - \left[\left(1 - \frac{1}{2}\alpha_R^2\right) + \left(1 - \frac{1}{2}\alpha_L^2\right) \right], \label{cond:consistency}   
\end{align} 
which ensures the $\gamma_2$ we construct from $\gamma_3$ satisfies the correct boundary conditions at $t = \pm 1$.  
By the definition of $\tilde \gamma$ and $\bar \gamma$ we have 
\begin{align*} 
\frac{1}{2}\int_{-1}^1 \left(\tilde \gamma_3^{\prime}(t)\right)^2 dt < \frac{1}{2}\int_{-1}^1 \left(\bar \gamma_3^{\prime}(t)\right)^2 dt = 2- \left[\left(1 - \frac{1}{2}\alpha_R^2\right) + \left(1 - \frac{1}{2}\alpha_L^2\right)\right]. 
\end{align*} 
For the final step we may proceed in a fashion similar to \cite{ContiMaggi08}, inserting a smooth bump in the region $[1/3,2/3]$ to ensure that \eqref{cond:consistency} holds. For example, define 
\begin{align*} 
\gamma_3(t) & = \left\{ 
\begin{array}{ll} 
\tilde \gamma_3(t) + \lambda \rho\left(2t - 1\right) & t \in [1/3,2/3] \\ 
\tilde \gamma_3(t) & t \in [-1,1]\setminus [1/3,2/3], 
\end{array}
\right.
\end{align*} 
with $\lambda$ chosen such that \eqref{cond:consistency} holds; this is possible because the smoothing reduced the integral on the left-hand side of \eqref{cond:consistency}.  
Then, we may integrate \eqref{cond:FvKGamma} to define a suitable $\gamma_2(t)$. 
The correct choice of $\lambda$ is given by 
\begin{align*} 
\lambda^2 = \frac{1}{2\int_{1/3}^{2/3}\left(\rho^\prime(2t - 1)\right)^2 dt}\left[\frac{1}{2}\alpha_R^2 + \frac{1}{2}\alpha_L^2 - \frac{1}{2}\int_{-1}^1 \left(\tilde\gamma_3^{\prime}(t)\right)^2\right] \lesssim \max\left(\abs{\alpha_L}^2,\abs{\alpha_R}^2\right),   
\end{align*}
where the last inequality followed since \eqref{ineq:gammaerror} holds for $\tilde \gamma_3$. 
Hence $\lambda \lesssim \max(\abs{\alpha_R},\abs{\alpha_L})$, which by the way we define $\gamma$ from $\tilde \gamma_3$ implies $\gamma_3$ also satisfies \eqref{ineq:gammaerror}. 
Inequality \eqref{ineq:etaomega1} follows from \eqref{ineq:gamma2} and \eqref{ineq:gamma2prime}. 
Similarly, \eqref{ineq:etaomega2} follows from \eqref{ineq:gamma3} and \eqref{ineq:gamma3prime}. 
Finally, \eqref{ineq:etaomega3} follows from \eqref{ineq:etaomega1},\eqref{ineq:etaomega2} and \eqref{ineq:gamma3prime}. 
\end{proof} 

\noindent
{\it Step 3: Final construction and optimal length-scale:} \\
At each sharp fold in $(\bar w,\bar u)$ defined in Step 1, we add a minimal ridge as constructed in Step 2.
This is accomplished simply by demarking non-overlapping quadrilaterals whose diagonals are the sharp folds and applying Lemma \ref{lem:MinRidge} to each of these separately. See Figure \ref{fig:SmthedConstruction} for a graphical depiction of this procedure. 
Denote the new construction by $(\tilde w, \tilde u)$.
By the definition of $(\bar w,\bar u)$ given in Step 1, we have $\phi \sim \sqrt{\eta}$ in each fold.
The choice \eqref{def:sigma}, ensures that the energy of each fold given by Lemma \ref{lem:MinRidge} (neglecting the vertices) matches that of the optimal scaling for an isolated fold \eqref{ineq:optEnergy}. 
The corresponding requirement expressed in \eqref{cond:sigma} becomes \eqref{cond:2d1b}. 
Technically, $\phi$ and $\sigma$ vary depending on what fold is being regularized, however they cannot differ from one another by more than a constant factor, so we can basically treat them as being global constants for our purposes. 
Another small detail which arises here is that $(\tilde w,\tilde u)$ has a slightly different bonded region than $(\bar w,\bar u)$, since the minimal ridge lifts the film from the substrate near the edges of the square $[0,\theta l]^2$. 
Since the adjustment is $O(\sigma l)$, we may enlarge the square $[0,\theta l]^2$ by $O(\sigma)$ on the sides to ensure that after the minimal ridges are added the test function satisfies $\abs{\Omega} = \theta$; we omit the details. 
By adjusting the constant in \eqref{cond:sigma}, we may ensure that this does not change our upper bound on the energy. 
 
\begin{center}
\begin{figure}[hbpt] 
\begin{picture}(150,150)(-100,0)
\scalebox{1.0}{\includegraphics{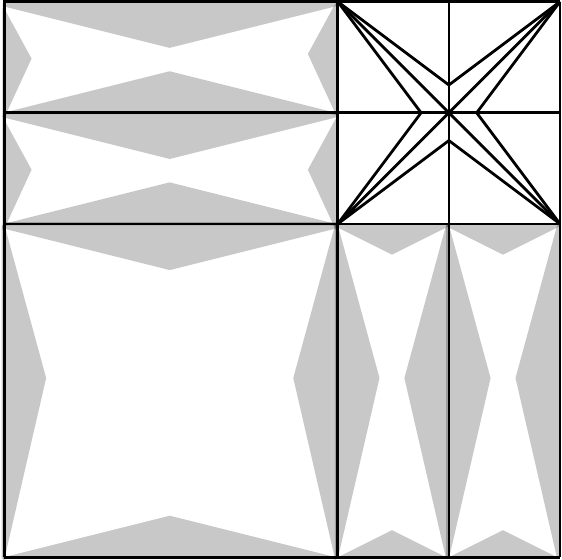}}
\end{picture}
\caption{A representative sample of quadrilaterals which can be added to apply Lemma \ref{lem:MinRidge}. Of course, similar quadrilaterals must also be added in the upper right corner.}
\label{fig:SmthedConstruction}
\end{figure}
\end{center}

An important detail which Lemma \ref{lem:MinRidge} does not address is how to handle the vertices where the folds connect; $(\tilde w,\tilde u)$ is continuous there but $\grad \tilde u$ may not be, resulting in infinite bending energy. 
Hence, $(\tilde w, \tilde u)$ must be corrected at the vertices as in Section 3 of \cite{ContiMaggi08} in order to have finite bending energy and achieve the desired scaling law.   
Denoting the vertices where the folds connect as $v_i$, $\Omega_\sigma^\star := \cup B(v_i,\sigma)$ and $\Omega_\sigma := \cup B(v_i,2\sigma)$, where $B(x_0,r) := \set{x\in \Real^2: \abs{x_0 - x} < r}$. We can assume the balls $B(v_i,2\sigma)$ are disjoint by, if necessary, adjusting the constant in \eqref{cond:sigma} (and therefore \eqref{cond:2d1b}). Away from $\Omega_\sigma^\star$, $(\tilde w,\tilde u)$ is smooth. 
Let $\rho \in C^\infty_0([-2,2];[0,1])$ with $\rho(x) = 1$ if $\abs{x} < 1$ and define our final deformation as 
\begin{align*} 
(w^h,u^h) := \sum_{i}\rho\left(\frac{\abs{x - v_i}}{\sigma}\right) (\bar w(v_i),\bar u(v_i)) +  \sum_{i}\left(1 - \rho\left(\frac{\abs{x - v_i}}{\sigma}\right)\right) (\tilde w(x),\tilde u(x)). 
\end{align*}
The following estimates follow from the definition of $(w^h,u^h)$, Lemma \ref{lem:MinRidge} and the fact that $\phi \sim \sqrt{\eta}$:
\begin{align*} 
\frac{1}{\sqrt{\eta}}\norm{\grad \bar w - \grad w^h}_{L^\infty([0,l]^2)} + \norm{\grad \bar u - \grad u^h}_{L^\infty([0,l]^2)} & \lesssim \sqrt{\eta} \\ 
\norm{D^2 u^h}_{L^\infty([0,l]^2)} & \lesssim \frac{\sqrt{\eta}}{\sigma}. 
\end{align*} 
It remains to estimate the contribution to the energy in the set $\Omega_\sigma$.  
The membrane energy is estimated by
\begin{align} 
\alpha_m h \int_{\Omega_\sigma} \abs{e(w^h) + \frac{1}{2}\grad u^h \otimes \grad u^h - \eta I}^2 dxdy \lesssim \alpha_m h \eta^2 \sigma^2 \lesssim h^3 \eta,  \label{ineq:OmegaSig1}
\end{align}
where the last inequality followed from \eqref{def:sigma}. 
The bending energy is estimated by
\begin{align} 
h^3 \int_{\Omega_\sigma} \abs{D^2u^h}^2 dx \lesssim  h^3\eta. \label{ineq:OmegaSig2}
\end{align}
Adding the contributions \eqref{ineq:optEnergy}, \eqref{ineq:substrateEst}, \eqref{ineq:OmegaSig1} and \eqref{ineq:OmegaSig2}, we get that over a single cell the total energy of $(w^h,u^h)$ is given by  
\begin{align*} 
E_{ND}[w^h,u^h] \lesssim \alpha_m^{1/6} \eta^{7/6} l^{1/3} h^{8/3} + \alpha_s \eta^2 l^3 + h^3\eta. 
\end{align*}
Summing over all $l^{-2}$ periodic cells gives a total energy of (again using equality in Cauchy-Schwarz as in the 1D case), 
\begin{align*} 
E_{ND}[w^h,u^h] \lesssim \alpha_m^{1/6} \eta^{7/6} l^{-5/3}h^{8/3} + \alpha_s \eta^2 l +  h^3\eta l^{-2}.   
\end{align*}
Optimizing $l$ between the first two terms gives
\begin{align*} 
l \sim \frac{\alpha_m^{1/16} h}{\eta^{5/16} \alpha_s^{3/8}} := l_2,    
\end{align*}
which implies
\begin{align*} 
E_{ND}[w^h,u^h] & \lesssim \alpha_m^{1/16} \eta^{27/16}\alpha_s^{5/8} h +  h \eta^{26/16} \alpha_s^{6/8} \alpha_m^{-1/8}. 
\end{align*}
Condition \eqref{cond:2d1b} then implies 
\begin{align*}
E_{ND}[w^h,u^h] & \lesssim \alpha_m^{1/16}\eta^{27/16}\alpha_s^{5/8} h.
\end{align*} 
Note that to apply Lemma \ref{lem:MinRidge} we require \eqref{cond:sigma} and \eqref{cond:inplane}; \eqref{cond:sigma} becomes $\sigma < c(\theta)l_2$ (for some $c<1$ depending on $\theta$), which is \eqref{cond:2d1b}. 
The condition \eqref{cond:2d1a} comes from the requirement that $l_2 < 1$ so that our construction fits inside unit the periodic square.  
Condition \eqref{cond:inplane} becomes (since $\phi \sim \sqrt{\eta}$), 
\begin{align*} 
\eta^2 \lesssim \frac{\sigma^{2/3}}{l_2^{2/3}},  
\end{align*} 
which becomes \eqref{cond:2d2} after simplifying. This now completes the proof of Theorem \ref{thm:2d}
\end{proof}

\vfill\eject
\bibliographystyle{plain}
\bibliography{materials}

\end{document}